\documentclass[oneside,english,reqno]{amsart}
\usepackage[T1]{fontenc}
\usepackage[latin9]{inputenc}
\usepackage{babel}
\usepackage{refstyle}
\usepackage{float}
\usepackage{mathrsfs}
\usepackage{amstext}
\usepackage{amsthm}
\usepackage{amssymb}
\usepackage{graphicx}
\usepackage{esint}
\usepackage[all]{xy}
\PassOptionsToPackage{normalem}{ulem}
\usepackage{ulem}
\usepackage[unicode=true,pdfusetitle,
 bookmarks=true,bookmarksnumbered=false,bookmarksopen=false,
 breaklinks=false,pdfborder={0 0 1},backref=false,colorlinks=false]
 {hyperref}
\hypersetup{
 colorlinks=true,citecolor=blue,linkcolor=blue,linktocpage=true}

\makeatletter

\AtBeginDocument{\providecommand\lemref[1]{\ref{lem:#1}}}
\AtBeginDocument{\providecommand\thmref[1]{\ref{thm:#1}}}
\AtBeginDocument{\providecommand\corref[1]{\ref{cor:#1}}}
\AtBeginDocument{\providecommand\secref[1]{\ref{sec:#1}}}
\AtBeginDocument{\providecommand\defref[1]{\ref{def:#1}}}
\AtBeginDocument{\providecommand\figref[1]{\ref{fig:#1}}}
\AtBeginDocument{\providecommand\propref[1]{\ref{prop:#1}}}
\AtBeginDocument{\providecommand\remref[1]{\ref{rem:#1}}}
\providecommand{\tabularnewline}{\\}
\RS@ifundefined{subref}
  {\def\RSsubtxt{section~}\newref{sub}{name = \RSsubtxt}}
  {}
\RS@ifundefined{thmref}
  {\def\RSthmtxt{theorem~}\newref{thm}{name = \RSthmtxt}}
  {}
\RS@ifundefined{lemref}
  {\def\RSlemtxt{lemma~}\newref{lem}{name = \RSlemtxt}}
  {}

\numberwithin{equation}{section}
\numberwithin{figure}{section}
\numberwithin{table}{section}
\usepackage{enumitem}		
\theoremstyle{plain}
\newtheorem{thm}{\protect\theoremname}[section]
  \theoremstyle{definition}
  \newtheorem{defn}[thm]{\protect\definitionname}
  \theoremstyle{remark}
  \newtheorem{rem}[thm]{\protect\remarkname}
  \theoremstyle{plain}
  \newtheorem{lem}[thm]{\protect\lemmaname}
  \theoremstyle{plain}
  \newtheorem{cor}[thm]{\protect\corollaryname}
  \theoremstyle{plain}
  \newtheorem{prop}[thm]{\protect\propositionname}
  \theoremstyle{remark}
  \newtheorem*{notation*}{\protect\notationname}
  \theoremstyle{plain}
  \newtheorem{question}[thm]{\protect\questionname}
  \theoremstyle{definition}
  \newtheorem{example}[thm]{\protect\examplename}
  \theoremstyle{remark}
  \newtheorem*{acknowledgement*}{\protect\acknowledgementname}

\usepackage{needspace}
\usepackage{bbm}
\allowdisplaybreaks

\newref{lem}{refcmd={Lemma \ref{#1}}}
\newref{thm}{refcmd={Theorem \ref{#1}}}
\newref{cor}{refcmd={Corollary \ref{#1}}}
\newref{sec}{refcmd={Section \ref{#1}}}
\newref{sub}{refcmd={Section \ref{#1}}}
\newref{chap}{refcmd={Chapter \ref{#1}}}
\newref{prop}{refcmd={Proposition \ref{#1}}}
\newref{exa}{refcmd={Example \ref{#1}}}
\newref{tab}{refcmd={Table \ref{#1}}}
\newref{rem}{refcmd={Remark \ref{#1}}}
\newref{def}{refcmd={Definition \ref{#1}}}
\newref{fig}{refcmd={Figure \ref{#1}}}
\newref{claim}{refcmd={Claim \ref{#1}}}

\@ifundefined{showcaptionsetup}{}{%
 \PassOptionsToPackage{caption=false}{subfig}}
\usepackage{subfig}
\AtBeginDocument{
  
}

\makeatother

  \providecommand{\acknowledgementname}{Acknowledgement}
  \providecommand{\corollaryname}{Corollary}
  \providecommand{\definitionname}{Definition}
  \providecommand{\examplename}{Example}
  \providecommand{\lemmaname}{Lemma}
  \providecommand{\notationname}{Notation}
  \providecommand{\propositionname}{Proposition}
  \providecommand{\questionname}{Question}
  \providecommand{\remarkname}{Remark}
\providecommand{\theoremname}{Theorem}

\begin{document}

\title{Nonuniform sampling, reproducing kernels, and the associated Hilbert
spaces}

\author{Palle Jorgensen and Feng Tian}

\address{(Palle E.T. Jorgensen) Department of Mathematics, The University
of Iowa, Iowa City, IA 52242-1419, U.S.A. }

\email{palle-jorgensen@uiowa.edu}

\urladdr{http://www.math.uiowa.edu/\textasciitilde{}jorgen/}

\address{(Feng Tian) Department of Mathematics, Hampton University, Hampton,
VA 23668, U.S.A.}

\email{feng.tian@hamptonu.edu}



\maketitle
\pagestyle{myheadings}
\markright{}
\begin{abstract}
In a general context of positive definite kernels $k$, we develop
tools and algorithms for sampling in reproducing kernel Hilbert space
$\mathscr{H}$ (RKHS). With reference to these RKHSs, our results
allow inference from samples; more precisely, reconstruction of an
``entire'' (or global) signal, a function $f$ from $\mathscr{H}$,
via generalized interpolation of $f$ from partial information obtained
from carefully chosen distributions of sample points. We give necessary
and sufficient conditions for configurations of point-masses $\delta_{x}$
of sample-points $x$ to have finite norm relative to the particular
RKHS $\mathscr{H}$ considered. When this is the case, and the kernel
$k$ is given, we obtain an induced positive definite kernel $\left\langle \delta_{x},\delta_{y}\right\rangle _{\mathscr{H}}$.
We perform a comparison, and we study when this induced positive definite
kernel has $l^{2}$ rows and columns. The latter task is accomplished
with the use of certain symmetric pairs of operators in the two Hilbert
spaces, $l^{2}$ on one side, and the RKHS $\mathscr{H}$ on the other.
A number of applications are given, including to infinite network
systems, to graph Laplacians, to resistance metrics, and to sampling
of Gaussian fields. 
\end{abstract}

\tableofcontents{}

\section{Introduction}

In the theory of non-uniform sampling, one studies Hilbert spaces
consisting of signals, understood in a very general sense. One then
develops analytic tools and algorithms, allowing one to draw inference
for an ``entire'' (or global) signal from partial information obtained
from carefully chosen distributions of sample points. While the better
known and classical sampling algorithms (Shannon and others) are based
on interpolation, modern theories go beyond this. An early motivation
is the work of Henry Landau. In this setting, it is possible to make
precise the notion of ``average sampling rates'' in general configurations
of sample points. Our present study, turns the tables. We start with
the general axiom system of positive definite kernels and their associated
reproducing kernel Hilbert spaces (RKHSs), or \emph{relative} RKHSs.
With some use of metric geometry and of spectral theory for operators
in Hilbert space, we are then able to obtain sampling theorems for
a host of non-uniform point configurations. The modern theory of non-uniform
sampling is vast, and it specializes into a variety of sub-areas.
The following papers (and the literature cited there) will give an
idea of the diversity of points of view: \cite{MR3172778,MR3189198,MR3268657,MR2451766,MR0222554}.

A reproducing kernel Hilbert space (RKHS) is a Hilbert space $\mathscr{H}$
of functions on a prescribed set, say $V$, with the property that
point-evaluation for functions $f\in\mathscr{H}$ is continuous with
respect to the $\mathscr{H}$-norm. They are called kernel spaces,
because, for every $x\in V$, the point-evaluation for functions $f\in\mathscr{H}$,
$f\left(x\right)$ must then be given as a $\mathscr{H}$-inner product
of $f$ and a vector $k_{x}$ in $\mathscr{H}$; called the kernel,
i.e., $f\left(x\right)=\left\langle k_{x},f\right\rangle _{\mathscr{H}}$,
$\forall f\in\mathscr{H}$, $x\in V$. 

There is a related reproducing kernel notion called ``relative:\textquotedblright{}
This means that increments have kernel representations. In detail:
Consider functions $f$ in $\mathscr{H}$, but suppose instead that,
for every pair of points $x,y$ in $V$, each of the differences $f\left(x\right)-f\left(y\right)$
can be represented by a kernel from $\mathscr{H}$. We then say that
$\mathscr{H}$ is a relative RKHS. We shall study both in our paper.
The \textquotedblleft relative\textquotedblright{} variant is of more
recent vintage, and it is used in the study of electrical networks
(voltage differences, see \lemref{dipp}); and in analysis of Gaussian
processes such as Gaussian fields, sect \ref{sub:bm} and \ref{sub:Bridge}.

The RKHSs have been studied extensively since the pioneering papers
by Aronszajn \cite{Aro43,Aro48}. They further play an important role
in the theory of partial differential operators (PDO); for example
as Green's functions of second order elliptic PDOs \cite{Nel57,HKL14}.
Other applications include engineering, physics, machine-learning
theory \cite{KH11,SZ09,CS02}, stochastic processes \cite{AD93,ABDdS93,AD92,AJSV13,AJV14},
numerical analysis, and more \cite{MR2089140,MR2607639,MR2913695,MR2975345,MR3062405,MR3091062,MR3101840,MR3201917,Shawe-TaylorCristianini200406,SchlkopfSmola200112}.
But the literature so far has focused on the theory of kernel functions
defined on continuous domains, either domains in Euclidean space,
or complex domains in one or more variables. For these cases, the
Dirac $\delta_{x}$ distributions do not have finite $\mathscr{H}$-norm.
But for RKHSs over discrete point distributions, it is reasonable
to expect that the Dirac $\delta_{x}$ functions will in fact have
finite $\mathscr{H}$-norm.

An illustration from \emph{neural networks}: An Extreme Learning Machine
(ELM) is a neural network configuration in which a hidden layer of
weights are randomly sampled \cite{RW06}, and the object is then
to determine analytically resulting output layer weights. Hence ELM
may be thought of as an approximation to a network with infinite number
of hidden units.

The main results in our paper include \thmref{pm}, \corref{sp},
\thmref{BB}, \thmref{dd}, \corref{CD}, and \thmref{mu} where we
give necessary and sufficient conditions for the point-masses to have
finite norm relative to the particular RKHS $\mathscr{H}$ considered.
When this is the case, we obtain an induced positive definite kernel
$\left\langle \delta_{x},\delta_{y}\right\rangle _{\mathscr{H}}$.
In \secref{SPA}, we study when this induced positive definite kernel
has $l^{2}$ rows and columns. The latter task is accomplished with
the use of certain symmetric pairs of operators in the two Hilbert
spaces, $l^{2}$ on one side, and the RKHS $\mathscr{H}$ on the other.
In \secref{mspa}, we study the cases when the associated symmetric
pair is maximal. The results from Sections \ref{sec:SPA}-\ref{sec:mspa}
are then applied in \secref{egs} to the study of admissible distributions
of discrete sample points for Brownian motion, and for related Gaussian
fields. We have a separate subsection \ref{sub:bion} discussing the
RKHS constructed canonically from the binomial coefficients.

\section{Reproducing kernel Hilbert spaces (RKHSs), and relative RKHSs}

Here we consider \emph{the discrete case}, i.e., RKHSs of functions
defined on a prescribed countable infinite discrete set $V$. We are
concerned with a characterization of those RKHSs $\mathscr{H}$ which
contain the Dirac masses $\delta_{x}$ for all points $x\in V$. Of
the examples and applications where this question plays an important
role, we emphasize three: (i) \emph{discrete Brownian motion-Hilbert
spaces}, i.e., discrete versions of the Cameron-Martin Hilbert space;
(ii) \emph{energy-Hilbert spaces} corresponding to graph-Laplacians;
and finally (iii) RKHSs generated by \emph{binomial coefficients}. 

In general when reproducing kernels and their Hilbert spaces are used,
one ends up with functions on a suitable set, and so far we feel that
the dichotomy \emph{discrete} vs \emph{continuous} has not yet received
sufficient attention. After all, a choice of sampling points in relevant
optimization models based on kernel theory suggests the need for a
better understanding of point masses as they are accounted for in
the RKHS at hand. In broad outline, this is a leading theme in our
paper.

The two definitions below, and \lemref{hf} are valid more generally
for the setting when $V$ is an arbitrary set. But we have nonetheless
restricted our focus to the case when $V$ is assumed countably infinite.
The reason for this will become evident in \defref{dmp}, and in \lemref{proj2},
\corref{proj1}, and \thmref{pm}, to follow.
\begin{defn}
Let $V$ be a countable and infinite set, and $\mathscr{F}\left(V\right)$
the set of all \emph{finite} subsets of $V$. A function $k:V\times V\longrightarrow\mathbb{C}$
is said to be \emph{positive definite}, if 
\begin{equation}
\underset{\left(x,y\right)\in F\times F}{\sum\sum}k\left(x,y\right)\overline{c_{x}}c_{y}\geq0\label{eq:pd1}
\end{equation}
holds for all coefficients $\{c_{x}\}_{x\in F}\subset\mathbb{C}$,
and all $F\in\mathscr{F}\left(V\right)$. 
\end{defn}

\begin{defn}
\label{def:d1}Fix a set $V$, countable infinite. 
\begin{enumerate}
\item For all $x\in V$, set 
\begin{equation}
k_{x}:=k\left(\cdot,x\right):V\longrightarrow\mathbb{C}\label{eq:pd2}
\end{equation}
as a function on $V$. 
\item Let $\mathscr{H}:=\mathscr{H}\left(k\right)$ be the Hilbert-completion
of the $span\left\{ k_{x}\mid x\in V\right\} $, with respect to the
inner product 
\begin{equation}
\left\langle \sum c_{x}k_{x},\sum d_{y}k_{y}\right\rangle _{\mathscr{H}}:=\sum\sum\overline{c_{x}}d_{y}k\left(x,y\right)\label{eq:pd3}
\end{equation}
modulo the subspace of functions of zero $\mathscr{H}$-norm. $\mathscr{H}$
is then a reproducing kernel Hilbert space (RKHS), with the reproducing
property:
\begin{equation}
\left\langle k_{x},\varphi\right\rangle _{\mathscr{H}}=\varphi\left(x\right),\;\forall x\in V,\quad\forall\varphi\in\mathscr{H}.\label{eq:pd31}
\end{equation}

\item If $F\in\mathscr{F}\left(V\right)$, set $\mathscr{H}_{F}=\text{closed\:\ span}\{k_{x}\}_{x\in F}\subset\mathscr{H}$,
(closed is automatic if $F$ is finite.) And set 
\begin{equation}
P_{F}:=\text{the orthogonal projection onto \ensuremath{\mathscr{H}_{F}}}.\label{eq:pd4}
\end{equation}

\item For $F\in\mathscr{F}\left(V\right)$, set 
\begin{equation}
K_{F}:=\left(k\left(x,y\right)\right)_{\left(x,y\right)\in F\times F}\label{eq:pd5}
\end{equation}
as a $\#F\times\#F$ matrix. 
\end{enumerate}
\end{defn}
\begin{rem}
The summations in (\ref{eq:pd3}) are all finite. Starting with finitely
supported summations in (\ref{eq:pd3}), the RKHS $\mathscr{H}\left(=\mathscr{H}\left(k\right)\right)$
is then obtained by Hilbert space completion. We use physicists' convention,
so that the inner product is conjugate linear in the first variable,
and linear in the second variable.
\end{rem}
The following result is known; and it follows from the definitions
above. 
\begin{lem}
\label{lem:hf}Let $k:V\times V\longrightarrow\mathbb{C}$ be positive
definite, and let $\mathscr{H}$ be the corresponding RKHS. Let $f$
be an arbitrary function on $V$. Then $f$ is in $\mathscr{H}$ if
and only if there is a constant $C=C_{f}<\infty$ such that, for every
finitely supported function $\xi:V\longrightarrow\mathbb{C}$, we
have the estimate
\[
\left|\sum_{x\in V}\overline{\xi\left(x\right)}f\left(x\right)\right|^{2}\leq C\underset{x,y\in V}{\sum\sum}\overline{\xi\left(x\right)}\xi\left(y\right)k\left(x,y\right)
\]
with the constant $C=C_{f}$ independent of $\xi$.
\end{lem}
It follows from the above that reproducing kernel Hilbert spaces (RKHSs)
arise from a given positive definite kernel $k$, a corresponding
pre-Hilbert form; and then a Hilbert-completion. The question arises:
\textquotedblleft What are the functions in the completion?\textquotedblright{}
Now, before completion, the functions are as specified in \defref{d1},
but the Hilbert space completions are subtle; they are classical Hilbert
spaces of functions, not always transparent from the naked kernel
$k$ itself. Examples of classical RKHSs: Hardy spaces or Bergman
spaces (for complex domains), Sobolev spaces and Dirichlet spaces
\cite{MR3054607,MR2892621,MR2764237} (for real domains, or for fractals),
band-limited $L^{2}$ functions (from signal analysis), and Cameron-Martin
Hilbert spaces from Gaussian processes (in continuous time domain).

Our focus here is on discrete analogues of the classical RKHSs from
real or complex analysis. These discrete RKHSs in turn are dictated
by applications, and their features are quite different from those
of their continuous counterparts.
\begin{defn}
\label{def:dmp}The RKHS $\mathscr{H}=\mathscr{H}\left(k\right)$
is said to have the \emph{discrete mass} property ($\mathscr{H}$
is called a \emph{discrete RKHS}), if $\delta_{x}\in\mathscr{H}$,
for all $x\in V$. Here, $\delta_{x}\left(y\right)=\begin{cases}
1 & \text{if }x=y\\
0 & \text{if \ensuremath{x\neq y}}
\end{cases}$, i.e., the Dirac mass at $x\in V$. \end{defn}
\begin{lem}
\label{lem:proj1}\label{lem:proj}Let $F\in\mathscr{F}\left(V\right)$,
$x_{1}\in F$. Assume $\delta_{x_{1}}\in\mathscr{H}$. Then 
\begin{equation}
P_{F}\left(\delta_{x_{1}}\right)\left(\cdot\right)=\sum_{y\in F}\left(K_{F}^{-1}\delta_{x_{1}}\right)\left(y\right)k_{y}\left(\cdot\right).\label{eq:pd6}
\end{equation}
\end{lem}
\begin{proof}
We check that 
\begin{equation}
\delta_{x_{1}}-\sum_{y\in F}\left(K_{F}^{-1}\delta_{x_{1}}\right)\left(y\right)k_{y}\left(\cdot\right)\in\mathscr{H}_{F}^{\perp}.\label{eq:pd7}
\end{equation}
The remaining part follows easily from this. 

(The notation $\left(\mathscr{H}_{F}\right)^{\perp}$ stands for orthogonal
complement, also denoted $\mathscr{H}\ominus\mathscr{H}_{F}=\left\{ \varphi\in\mathscr{H}\:\big|\:\left\langle f,\varphi\right\rangle _{\mathscr{H}}=0,\;\forall f\in\mathscr{H}_{F}\right\} $.)\end{proof}
\begin{rem}
A slight abuse of notations: We make formally sense of the expressions
for $P_{F}(\delta_{x})$ in (\ref{eq:pd6}) even in the case when
$\delta_{x}$ might not be in $\mathscr{H}$. For all finite $F$,
we showed that $P_{F}(\delta_{x})\in\mathscr{H}$. But for $\delta_{x}$
be in $\mathscr{H}$, we must have the additional boundedness assumption
(\ref{eq:d3}) satisfied; see \thmref{del}.\end{rem}
\begin{lem}
\label{lem:proj2}Let $F\in\mathscr{F}\left(V\right)$, $x_{1}\in F$,
then 
\begin{equation}
\left(K_{F}^{-1}\delta_{x_{1}}\right)\left(x_{1}\right)=\left\Vert P_{F}\delta_{x_{1}}\right\Vert _{\mathscr{H}}^{2}.\label{eq:pd8}
\end{equation}
\end{lem}
\begin{proof}
Setting $\zeta^{\left(F\right)}:=K_{F}^{-1}\left(\delta_{x_{1}}\right)$,
we have 
\[
P_{F}\left(\delta_{x_{1}}\right)=\sum_{y\in F}\zeta^{\left(F\right)}\left(y\right)k_{F}\left(\cdot,y\right)
\]
and for all $z\in F$, 
\begin{eqnarray}
\underset{\zeta^{\left(F\right)}\left(x_{1}\right)}{\underbrace{\sum_{z\in F}\zeta^{\left(F\right)}\left(z\right)P_{F}\left(\delta_{x_{1}}\right)\left(z\right)}} & = & \sum_{F}\sum_{F}\zeta^{\left(F\right)}\left(z\right)\zeta^{\left(F\right)}\left(y\right)k_{F}\left(z,y\right)\label{eq:pd81}\\
 & = & \left\Vert P_{F}\delta_{x_{1}}\right\Vert _{\mathscr{H}}^{2}.\nonumber 
\end{eqnarray}
By \lemref{proj}, the LHS of (\ref{eq:pd81}) is given by 
\begin{eqnarray*}
\left\Vert P_{F}\delta_{x_{1}}\right\Vert _{\mathscr{H}}^{2} & = & \left\langle P_{F}\delta_{x_{1}},\delta_{x_{1}}\right\rangle _{\mathscr{H}}\\
 & = & \sum_{y\in F}\left(K_{F}^{-1}\delta_{x_{1}}\right)\left(y\right)\left\langle k_{y},\delta_{x_{1}}\right\rangle _{\mathscr{H}}\\
 & = & \left(K_{F}^{-1}\delta_{x_{1}}\right)\left(x_{1}\right)=K_{F}^{-1}\left(x_{1},x_{1}\right).
\end{eqnarray*}
\end{proof}
\begin{cor}
\label{cor:proj1}If $\delta_{x_{1}}\in\mathscr{H}$ (see \thmref{del}),
then 
\begin{equation}
\sup_{F\in\mathscr{F}\left(V\right)}\left(K_{F}^{-1}\delta_{x_{1}}\right)\left(x_{1}\right)=\left\Vert \delta_{x_{1}}\right\Vert _{\mathscr{H}}^{2}.\label{eq:pd9}
\end{equation}
\end{cor}
\begin{thm}
\label{thm:pm}\label{thm:del}Given $V$, $k:V\times V\rightarrow\mathbb{R}$
positive definite (p.d.). Let $\mathscr{H}=\mathscr{H}\left(k\right)$
be the corresponding RKHS. Assume $V$ is countably infinite. Then
the following three conditions (\ref{enu:d1})-(\ref{enu:d3}) are
equivalent; $x_{1}\in V$ is fixed:
\begin{enumerate}
\item \label{enu:d1}$\delta_{x_{1}}\in\mathscr{H}$; 
\item \label{enu:d2}$\exists C_{x_{1}}<\infty$ such that for all $F\in\mathscr{F}\left(V\right)$,
the following estimate holds:
\begin{equation}
\left|\xi\left(x_{1}\right)\right|^{2}\leq C_{x_{1}}\underset{\left(x,y\right)\in F\times F}{\sum\sum}\overline{\xi\left(x\right)}\xi\left(y\right)k\left(x,y\right)\label{eq:fm0}
\end{equation}

\item \label{enu:d3}For $F\in\mathscr{F}\left(V\right)$, set 
\begin{equation}
K_{F}=\left(k\left(x,y\right)\right)_{\left(x,y\right)\in F\times F}\label{eq:d2}
\end{equation}
as a $\#F\times\#F$ matrix. Then
\begin{equation}
\sup_{F\in\mathscr{F}\left(V\right)}\left(K_{F}^{-1}\delta_{x_{1}}\right)\left(x_{1}\right)<\infty.\label{eq:d3}
\end{equation}

\end{enumerate}
\end{thm}
\begin{proof}
For details, see \cite{2015arXiv150102310J,2015arXiv150104954J}.
See also \lemref{hf}.
\end{proof}
Following \cite{KZ96}, we say that $k$ is \emph{strictly positive}
iff  $\det K_{F}>0$ for all $F\in\mathscr{F}\left(V\right)$. 

\begin{figure}[H]
\begin{tabular}[t]{ccc}
\includegraphics[width=0.3\columnwidth]{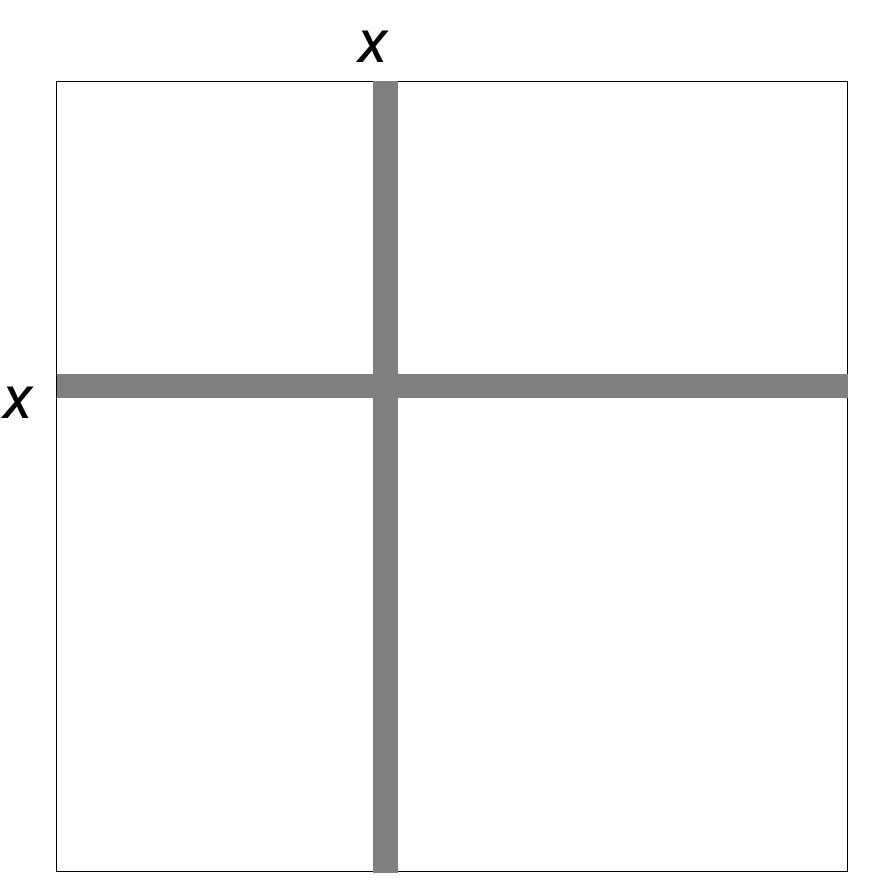} &  & \includegraphics[width=0.3\columnwidth]{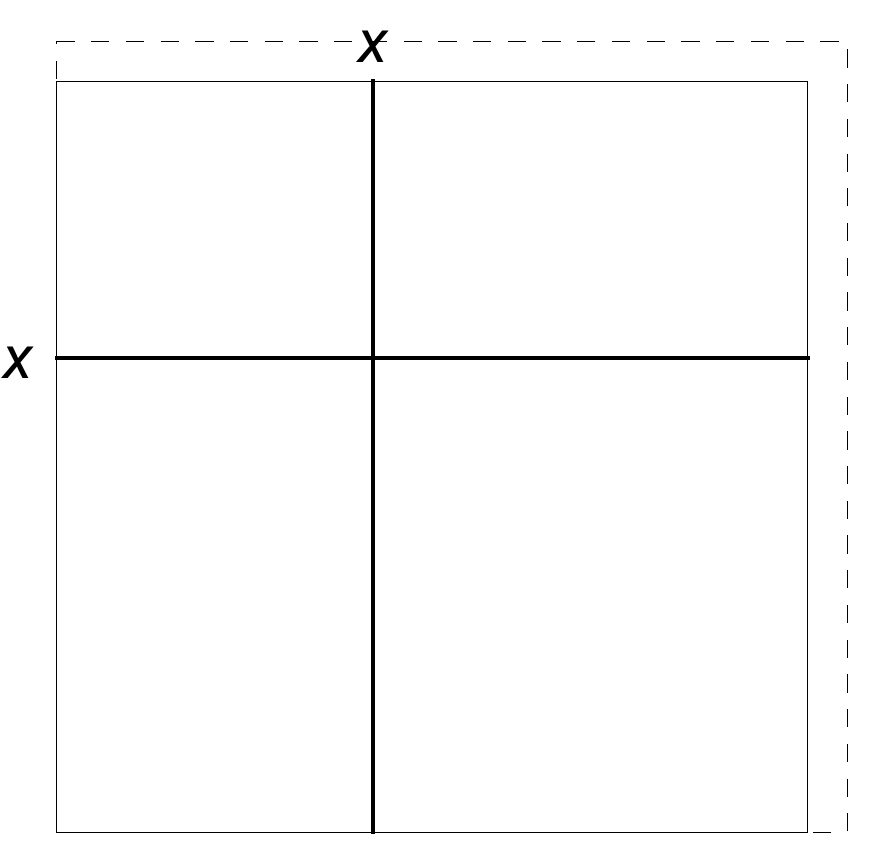}\tabularnewline
\end{tabular}

\caption{\label{fig:minor}The $\left(x,x\right)$ minors, $K_{F}\rightarrow K_{F}'$. }
\end{figure}

\begin{cor}
\label{cor:sp}Suppose $k:V\times V\rightarrow\mathbb{R}$ is strictly
positive. Set $D_{F}:=\det K_{F}$. If $x\in V$, and $F\in\mathscr{F}_{x}\left(V\right)$,
set $K'_{F}:=$ the minor in $K_{F}$ obtained by omitting row $x$
and column $x$, see \figref{minor}. Let $x\in V$. Then 
\begin{equation}
\delta_{x}\in\mathscr{H}\Longleftrightarrow\sup_{F\in\mathscr{F}_{x}\left(V\right)}\frac{D'_{F}}{D_{F}}<\infty.
\end{equation}

\end{cor}

\section{Sampling and point-masses of finite norm}

The results presented below hold both for the case of reproducing
kernel Hilbert spaces (RKHSs), and the parallel case of \emph{relative}
RKHSs. However, we shall state theorems only in the first case. The
reader will be able to formulate the results in the case of relative
RKHSs. The proofs in the relative case are the same but with slight
modifications \emph{mutatis mutandis}. 

An important special case of relative RKHSs is that of infinite networks
(or graphs) treated in the earlier literature.

\emph{Infinite vs finite graphs}. We study \textquotedblleft large
weighted graphs\textquotedblright{} (vertices $V$, edges $E$, and
weights as functions assigned on the edges $E$), and our motivation
derives from learning where \textquotedblleft learning\textquotedblright{}
is understood broadly to include (machine) learning of suitable probability
distribution, i.e., meaning learning from samples of training data.
Other applications of an analysis of weighted graphs include statistical
mechanics, such as infinite spin models, and large digital networks.
It is natural to ask then how one best approaches analysis on \textquotedblleft large\textquotedblright{}
systems. We propose an analysis via infinite weighted graphs. This
is so even if some of the questions in learning theory may in fact
refer to only \textquotedblleft large\textquotedblright{} finite graphs.

One reason for this (among others) is that \emph{statistical features}
in such an analysis are best predicted by consideration of probability
spaces corresponding to measures on infinite sample spaces. Moreover
the latter are best designed from consideration of infinite weighted
graphs, as opposed to their finite counterparts. Examples of statistical
features which are relevant even for finite samples is long-range
order; i.e., the study of correlations between distant sites (vertices),
and related phase-transitions, e.g., sign-flips at distant sites.
In designing efficient learning models, it is important to understand
the possible occurrence of unexpected long-range correlations; e.g.,
correlations between distant sites in a finite sample.

A second reason for the use of infinite sample-spaces is their use
in designing efficient sampling procedures. The interesting solutions
will often occur first as vectors in an infinite-dimensional reproducing-kernel
Hilbert space RKHS. Indeed, such RKHSs serve as powerful tools in
the solution of a kernel-optimization problems with penalty terms.
Once an optimal solution is obtained in infinite dimensions, one may
then proceed to study its restrictions to suitably chosen finite subgraphs.
See \cite{MR2573134,MR3074509,MR3285408,MR3167762,MR3390972}.
\begin{defn}
\label{def:nt}An infinite network consists of the following:
\begin{itemize}
\item $V$ a set of vertices, $\#V=\aleph_{0}$;
\item $E\subset V\times V\backslash\left\{ \mbox{diagonal}\right\} $, edges;
\item $c:E\longrightarrow\mathbb{R}_{+}$ a fixed symmetric function representing
conductance. 
\end{itemize}

We assume $\left(V,E,c\right)$ is connected, i.e., for $\forall x,y\in V$,
$\exists\left(x_{i}x_{i+1}\right)_{i=0}^{n-1}\in E$ s.t. $x_{0}=x$,
$x_{n}=y$. 

\end{defn}

\begin{defn}
\label{def:eh}Let $\left(V,E,c\right)$ be an infinite network. We
denote by $\mathscr{H}_{E}$ the \emph{energy Hilbert space}, where
\[
\mathscr{H}_{E}=\Big\{ f:V\longrightarrow\mathbb{C}\mid\left\Vert f\right\Vert _{\mathscr{H}_{E}}^{2}:=\frac{1}{2}\underset{\left(xy\right)\in E}{\sum\sum}c_{xy}\left|f\left(x\right)-f\left(y\right)\right|^{2}<\infty\Big\}.
\]
Given $f,g\in\mathscr{H}_{E}$, the inner product is 
\[
\left\langle f,g\right\rangle _{\mathscr{H}_{E}}=\frac{1}{2}\underset{\left(xy\right)\in E}{\sum\sum}c_{xy}(\overline{f\left(x\right)}-\overline{f\left(y\right)})(f\left(x\right)-f\left(y\right)).
\]

\end{defn}
The following two facts are well-known:
\begin{lem}
\label{lem:dipp}Let $V,E,\mathscr{H}_{E}$ be as in Definitions \ref{def:nt}
and \ref{def:eh}. Then
\begin{enumerate}
\item \label{enu:f1}$\mathscr{H}_{E}$ is a Hilbert space; and
\item \label{enu:f2}For $\forall x,y\in V$, $\exists!v_{xy}\in\mathscr{H}_{E}$
(called dipole) s.t. 
\begin{equation}
f\left(x\right)-f\left(y\right)=\left\langle v_{xy},f\right\rangle _{\mathscr{H}_{E}},\quad\forall f\in\mathscr{H}_{E}.\label{eq:di1}
\end{equation}

\end{enumerate}
\end{lem}
\begin{proof}
While this is in the literature, we will include a brief sketch. Part
(\ref{enu:f1}) is clear. To prove (\ref{enu:f2}), recall that it
is assumed that $\left(V,E\right)$ is connected; so given any pair
$x,y\in V$, $\exists n\in\mathbb{N}$, and $\left(x_{i}x_{i+1}\right)_{i=0}^{n-1}\in E$
s.t. $x_{0}=x$ and $x_{n}=y$. Then, for $\forall f\in\mathscr{H}_{E}$,
we have 
\begin{eqnarray*}
\left|f\left(y\right)-f\left(x\right)\right| & = & \left|\sum_{i=0}^{n-1}f\left(x_{i+1}\right)-f\left(x_{i}\right)\right|\\
 & = & \left|\sum_{i=0}^{n-1}\frac{1}{\sqrt{c_{x_{i}x_{i+1}}}}\sqrt{c_{x_{i}x_{i+1}}}\left(f\left(x_{i+1}\right)-f\left(x_{i}\right)\right)\right|\\
\left(\text{Schwarz}\right) & \leq & \left(\sum_{i=0}^{n-1}\frac{1}{c_{x_{i}x_{i+1}}}\right)^{\frac{1}{2}}\left(\sum_{i=0}^{n-1}c_{x_{i}x_{i+1}}\left|f\left(x_{i+1}\right)-f\left(x_{i}\right)\right|^{2}\right)^{\frac{1}{2}}\\
 & \leq & \mbox{Const}\cdot\left\Vert f\right\Vert _{\mathscr{H}_{E}}.
\end{eqnarray*}
The existence of $v_{xy}\in\mathscr{H}_{E}$ as asserted in (\ref{eq:di1})
now follows from an application of Riesz' theorem to the Hilbert space
$\mathscr{H}_{E}$. Also see \cite{MR3390972,JP10,2015arXiv150102310J}.\end{proof}
\begin{defn}
\label{def:fm}It will be convenient to choose a fixed base-point,
say $o\in V$, and set $v_{x}:=v_{xo}$. In this case, (\ref{eq:di1})
takes the form
\begin{equation}
f\left(x\right)-f\left(o\right)=\left\langle v_{xo},f\right\rangle _{\mathscr{H}_{E}},\quad\forall f\in\mathscr{H}_{E}.
\end{equation}
We say that $\mathscr{H}_{E}$ is a relative RKHS. The corresponding
positive definite kernel is as follows:
\[
k\left(x,y\right)=\left\langle v_{x},v_{y}\right\rangle _{\mathscr{H}_{E}},\quad\left(x,y\right)\in V\times V.
\]

We say that a given infinite network $\left(V,E,c,\mathscr{H}_{E}\right)$
as above has finite point-masses iff 
\begin{equation}
\delta_{x}\in\mathscr{H}_{E},\quad\forall x\in V.\label{eq:fm}
\end{equation}
\end{defn}
\begin{rem}
The condition (\ref{eq:fm}) will be automatic if for all $x\in V$,
\begin{equation}
\#\left\{ y\in V\mid\left(xy\right)\in E\right\} <\infty,\label{eq:fn}
\end{equation}
but the finite-point mass case holds in many examples where (\ref{eq:fn})
is not assumed, see \secref{SPA} below. \end{rem}
\begin{prop}
\label{prop:glap}Let $\left(V,E,c\right)$ and $\mathscr{H}_{E}$
be as in \defref{eh}. Assume condition (\ref{eq:fn}) is satisfied.
For functions $f$ on $V$, set 
\begin{equation}
\left(\Delta_{c}f\right)\left(x\right)=\sum_{y\sim x}c_{xy}\left(f\left(x\right)-f\left(y\right)\right).\label{eq:L1}
\end{equation}
Finally, let $\left\{ v_{x}\right\} _{x\in V\backslash\left\{ o\right\} }$
be a system of dipoles; see \defref{fm}. Set 
\begin{equation}
k_{x}\left(y\right):=\left\langle v_{y},v_{x}\right\rangle _{\mathscr{H}},\quad y\in V;\label{eq:L2}
\end{equation}
then 
\begin{equation}
\left(\Delta_{c}k_{x}\right)\left(y\right)=\delta_{x,y},\quad\forall x,y\in V\backslash\left\{ o\right\} .\label{eq:L3}
\end{equation}
\end{prop}
\begin{proof}
The verification of (\ref{eq:L3}) is a direct computation which we
leave to the reader. 

Because of (\ref{eq:L3}), one often says that $k$ (in (\ref{eq:L2}))
is a Green's function for the graph Laplacian $\Delta_{c}$ in (\ref{eq:L1}). 

For other applications of related semibounded selfadjoint operators,
see e.g., \cite{MR3167762}.
\end{proof}

\section{\label{sec:SPA}A symmetric pair of operators associated with a RKHS
having its point-masses of finite norm}

We now turn to the general case of positive definite kernels and the
case of RKHSs, and \emph{relative} RKHSs, such that the point-mass
condition (\ref{eq:fm}) is satisfied. We show that there is then
an associated and canonical symmetric pair of operators $\left(A,B\right)$: 
\begin{defn}
\label{def:dp}Let $k:V\times V\longrightarrow\mathbb{C}\left(\mbox{or }\mathbb{R}\right)$
be a positive definite kernel, and $\mathscr{H}$ be the corresponding
RKHS as above. If (\ref{eq:fm0}) holds, i.e., $\mathscr{H}$ has
the finite-mass property (Def. \ref{def:dmp}), then we get a dual
pair of operators as follows (see Fig \ref{fig:dp}):

$A:l^{2}\left(V\right)\longrightarrow\mathscr{H}\left(=\mathscr{H}\left(k\right)\right)$,
$\mathscr{D}\left(A\right)=span\left\{ \delta_{x}\right\} $ dense
in $l^{2}\left(V\right)$, with 
\[
A\delta_{x}=\delta_{x}\in\mathscr{H};
\]

$B:\mathscr{H}\longrightarrow l^{2}\left(V\right)$, $\mathscr{D}\left(B\right)=span\left\{ k_{x}\right\} $
dense in $\mathscr{H}$, where
\begin{itemize}
\item[] Case 1. RKHS:
\begin{equation}
Bk_{x}=\delta_{x}\label{eq:B1}
\end{equation}

\item[] Case 2. Relative RKHS (\defref{fm}): 
\begin{equation}
Bk_{x}=\delta_{x}-\delta_{o}.\label{eq:B2}
\end{equation}

\end{itemize}
\end{defn}
\begin{figure}[H]
\[
\xymatrix{l^{2}\left(V\right)\ar@/^{1.5pc}/[rr]^{A} &  & \mathscr{H}\left(=\mathscr{H}\left(k\right)\right)\ar@/^{1.5pc}/[ll]^{B}}
\]

\caption{\label{fig:dp}The pair of operators $\left(A,B\right)$.}
\end{figure}

\begin{prop}
\label{prop:AB}The system $\left(A,B\right)$ from \defref{dp} is
a symmetric pair, i.e., 
\begin{equation}
\left\langle Au,v\right\rangle _{\mathscr{H}}=\left\langle u,Bv\right\rangle _{l^{2}},\quad\forall u\in\mathscr{D}\left(A\right),\:\forall v\in\mathscr{D}\left(B\right).\label{eq:spa}
\end{equation}
\end{prop}
\begin{proof}
It suffices to consider real Hilbert spaces (the modifications needed
for the complex case are straightforward), in which case we have: 
\begin{enumerate}[label=]
\item $\left\langle u,v\right\rangle _{\mathscr{H}}$ the inner product
in $\mathscr{H}\left(=\mathscr{H}\left(k\right)\right)$;
\item $\left\langle \xi,\eta\right\rangle _{l^{2}}=\sum_{x\in V}\overline{\xi\left(x\right)}\eta\left(x\right)$; 
\item $\left\Vert \xi\right\Vert _{l^{2}}^{2}=\sum_{x\in V}\left|\xi\left(x\right)\right|^{2}<\infty$; 
\item $\left\Vert k_{x}\right\Vert _{\mathscr{H}}^{2}=k\left(x,x\right)$,
$\forall x\in V$. 
\end{enumerate}
To check (\ref{eq:spa}), it is enough to prove that 
\begin{equation}
\left\langle A\delta_{x},k_{y}\right\rangle _{\mathscr{H}}=\left\langle \delta_{x},Bk_{y}\right\rangle _{l^{2}},\quad\forall x,y\in V.\label{eq:spa1}
\end{equation}
Case 1. 
\[
\mbox{LHS}_{\left(\ref{eq:spa1}\right)}=\left\langle \delta_{x},k_{y}\right\rangle _{\mathscr{H}}=\delta_{xy}=\left\langle \delta_{x},\delta_{y}\right\rangle _{l^{2}}=\mbox{RHS}_{\left(\ref{eq:spa1}\right)}.
\]
Case 2. 
\[
\mbox{LHS}_{\left(\ref{eq:spa1}\right)}=\left\langle \delta_{x},k_{y}\right\rangle _{\mathscr{H}}=\delta_{x}\left(y\right)-\delta_{x}\left(o\right)=\mbox{RHS}_{\left(\ref{eq:spa1}\right)}.
\]
See (\ref{eq:B1})-(\ref{eq:B2}) in \defref{dp}.\end{proof}
\begin{notation*}[closure]
 Below we shall use the following terminology for the closure of
linear operators $\mathscr{H}_{1}\xrightarrow{\;T\;}\mathscr{H}_{2}$
where $T$ has dense domain in $\mathscr{H}_{1}$; and $dom\left(T^{*}\right)$
is assumed dense in $\mathscr{H}_{2}$. The graph of $T$ is 
\[
G\left(T\right)=\left\{ \begin{pmatrix}h\\
Th
\end{pmatrix}\mid h\in dom\left(T\right)\right\} \subseteq\begin{pmatrix}\underset{\oplus}{\mathscr{H}_{1}}\\
\mathscr{H}_{2}
\end{pmatrix}.
\]
The $\begin{pmatrix}\underset{\oplus}{\mathscr{H}_{1}}\\
\mathscr{H}_{2}
\end{pmatrix}$-closure of $G\left(T\right)$ is then the graph of the ``closure''
of $T$, written $\overline{T}$; in short, $\overline{G\left(T\right)}=G\left(\overline{T}\right)$.
We also have $T^{**}=\overline{T}$. \end{notation*}
\begin{cor}
\label{cor:ab}Let the operators $\left(A,B\right)$ be as above. 
\begin{enumerate}
\item We have $A\subset B^{*}$, and $B\subset A^{*}$; see Fig \ref{fig:adj}. 
\item Moreover, both operators below have dense domains, and are selfadjoint:

\begin{itemize}
\item[] $dom\left(A^{*}\overline{A}\right)$ is dense in $l^{2}$, and $A^{*}\overline{A}$
is s.a. in $l^{2}$; 
\item[] $dom\left(B^{*}\overline{B}\right)$ is dense in $\mathscr{H}$,
and $B^{*}\overline{B}$ is s.a. in $\mathscr{H}$. 
\end{itemize}
\item Using polar decomposition, we then get:
\begin{align}
\overline{A} & =V\left(A^{*}\overline{A}\right)^{\frac{1}{2}}=\left(\overline{A}A^{*}\right)^{\frac{1}{2}}V\label{eq:p1}\\
\overline{B} & =W\left(B^{*}\overline{B}\right)^{\frac{1}{2}}=\left(\overline{B}B^{*}\right)^{\frac{1}{2}}W\label{eq:p2}
\end{align}
with partial isometries $V:l^{2}\longrightarrow\mathscr{H}$, $W:\mathscr{H}\longrightarrow l^{2}$,
and 
\begin{align}
V^{*}V & =I_{l^{2}}-\mbox{Proj}\ker\left(\overline{A}\right),\;\mbox{and}\\
W^{*}W & =I_{\mathscr{H}}-\mbox{Proj}\ker\left(\overline{B}\right).
\end{align}

\end{enumerate}
\end{cor}
\begin{proof}
The conclusions here follow from the fundamentals regarding the polar
decomposition (factorization), in the setting of general unbounded
closable operators; see e.g., \cite{MR0052042,DS88b}.
\end{proof}
\begin{figure}[H]
\subfloat{$\xymatrix{l^{2}\left(V\right)\ar@/^{1.5pc}/[rr]^{A} &  & \mathscr{H}\left(=\mathscr{H}\left(k\right)\right)\ar@/^{1.5pc}/[ll]^{B}}
$}\qquad{}\subfloat{$\xymatrix{l^{2}\ar@/_{1.3pc}/[rr]_{B^{*}} &  & \mathscr{H}\ar@/_{1.3pc}/[ll]_{A^{*}}}
$}

\caption{\label{fig:adj}The symmetric pair $\left(A,B\right)$, with $\mathscr{D}\left(A\right)=span\left\{ \delta_{x}\right\} $,
and $\mathscr{D}\left(B\right)=span\left\{ k_{x}\right\} $. }
\end{figure}

\begin{rem}
\label{rem:aa}Since $A^{*}\overline{A}$ is selfadjoint in $l^{2}$,
it has a projection valued measure $P^{\left(A\right)}$. The following
property of $P^{\left(A\right)}$ shall be used below: If $\psi$
is a Borel function on $[0,\infty)$, then the functional calculus
operator $\psi\left(A^{*}\overline{A}\right)$ has the following representation
\[
\psi\left(A^{*}\overline{A}\right)=\int_{0}^{\infty}\psi\left(\lambda\right)P^{\left(A\right)}\left(d\lambda\right).
\]
Given $\xi\in l^{2}$, we therefore have:
\[
\xi\in dom\left(\psi\left(A^{*}\overline{A}\right)\right)\Longleftrightarrow\int_{0}^{\infty}\left|\psi\left(\lambda\right)\right|^{2}\left\Vert P^{\left(A\right)}\left(d\lambda\right)\xi\right\Vert _{l^{2}}^{2}<\infty;
\]
and in this case, 
\[
\left\Vert \psi\left(A^{*}\overline{A}\right)\xi\right\Vert _{l^{2}}^{2}=\int_{0}^{\infty}\left|\psi\left(\lambda\right)\right|^{2}\left\Vert P^{\left(A\right)}\left(d\lambda\right)\xi\right\Vert _{l^{2}}^{2}.
\]

\end{rem}

\begin{rem}
Let $\left(V,k,\mathscr{H}\right)$ be as in \defref{dp} and \thmref{BB};
i.e., we are assuming that $\delta_{x}\in\mathscr{H}$, $\forall x\in V$.
Let $\left(A,B\right)$ be the associated symmetric pair; see \propref{AB}.
Set 
\begin{equation}
H{\scriptstyle AR}\left(\mathscr{H},k\right):=\left\{ h\in\mathscr{H}\mid\left(A^{*}h\right)\left(x\right)=0,\;\forall x\in V\right\} =\ker\left(A^{*}\right),\label{eq:ha1}
\end{equation}
and let 
\begin{equation}
D{\scriptstyle EL}\left(\mathscr{H},k\right):=\mbox{the \ensuremath{\mathscr{H}}-closure of \ensuremath{span\left\{ \delta_{x}\mid x\in V\right\} .}}\label{eq:ha2}
\end{equation}
\end{rem}
\begin{lem}
\label{lem:split}We have the following orthogonal splitting:
\begin{equation}
\mathscr{H}=H{\scriptstyle AR}\left(\mathscr{H},k\right)\oplus D{\scriptstyle EL}\left(\mathscr{H},k\right)\label{eq:ha3}
\end{equation}
\end{lem}
\begin{proof}
Since $D{\scriptstyle EL}\left(\mathscr{H},k\right)=\left\{ \delta_{x}\mid x\in V\right\} ^{\perp\perp}$
where ``$\perp$'' refers to the inner product $\left\langle \cdot,\cdot\right\rangle _{\mathscr{H}}$,
we need only to show that 
\begin{equation}
\left\langle h,\delta_{x}\right\rangle _{\mathscr{H}}=0,\;\forall x\in V\Longleftrightarrow\left(A^{*}h\right)\left(x\right)=0,\;\forall x\in V.\label{eq:ha4}
\end{equation}
But the last eq. (\ref{eq:ha4}) follows from the duality in \propref{AB}:

If $\left\langle h,\delta_{x}\right\rangle _{\mathscr{H}}=0$, $\forall x\in V$;
then $h\in dom\left(A^{*}\right)$, and 
\[
\left(A^{*}h\right)\left(x\right)=\left\langle \delta_{x},A^{*}h\right\rangle _{l^{2}}=\left\langle A\delta_{x},h\right\rangle _{\mathscr{H}}=\left\langle \delta_{x},h\right\rangle _{\mathscr{H}}.
\]
The desired conclusion (\ref{eq:ha3}) is now immediate from this.\end{proof}
\begin{question}
Assume the point-mass property (Def. \ref{def:dmp}), i.e., $\delta_{x}\in\mathscr{H}$,
$\forall x\in V$. How do we compute the following two positive definite
(p.d.) kernels? The p.d. kernel $k$, with
\[
V\times V\ni\left(x,y\right)\longrightarrow k\left(x,y\right);
\]
and the induced p.d. kernel 
\[
V\times V\ni\left(x,y\right)\longrightarrow\left\langle \delta_{x},\delta_{y}\right\rangle _{\mathscr{H}}.
\]
Note that we are not assuming that $\delta_{x}\in dom\left(\overline{B}\right)$.
How to compute the kernels $k\left(\cdot,\cdot\right)$ and $\left\langle \delta_{x},\delta_{y}\right\rangle _{\mathscr{H}}$?
See details below.\end{question}
\begin{thm}
\label{thm:BB}Let $k:V\times V\longrightarrow\mathbb{C}\left(\mbox{or }\mathbb{R}\right)$
be given positive definite, and assume 
\begin{equation}
\delta_{x}\in\mathscr{H}\left(=\mathscr{H}\left(k\right)\right),\quad\forall x\in V,\label{eq:ff0}
\end{equation}
i.e., $\mathscr{H}$ is a RKHS with point masses. Let $\left(A,B\right)$
be the canonical symmetric pair; see \defref{dp}. Then 
\begin{equation}
k_{x}\in dom\left(B^{*}\overline{B}\right),\quad\forall x\in V.\label{eq:ff1}
\end{equation}
\end{thm}
\begin{proof}
Since $\overline{B}\,k_{x}=B\,k_{x}=\delta_{x}$, the desired conclusion
follows if we show that $\delta_{x}\in dom\left(B^{*}\right)$, i.e.,
$\forall x\in V$, $\exists C_{x}<\infty$ s.t. 
\begin{equation}
\left|\left\langle \delta_{x},B\,f\right\rangle _{l^{2}}\right|^{2}\leq C_{x}\left\Vert f\right\Vert _{\mathscr{H}}^{2},\quad\forall f\in\mathscr{D}\left(B\right).\label{eq:ff2}
\end{equation}
Fix $x=x_{0}\in V$. Now set $f=\sum_{y}\xi_{y}k_{y}\in\mathscr{D}\left(B\right)$
(finite sum), then 
\[
\left\langle \delta_{x_{0}},B\,f\right\rangle _{l^{2}}=\left\langle \delta_{x_{0}},\sum\nolimits _{y}\xi_{y}\delta_{y}\right\rangle _{l^{2}}=\xi_{x_{0}}.
\]
But by (\ref{eq:ff0}) and \lemref{hf}, we get the desired constant
$C_{x}<\infty$ s.t. 
\[
\left|\xi_{x_{0}}\right|^{2}\leq C_{x_{0}}\underset{\left(y,z\right)\in V\times V}{\sum\sum}\xi_{y}\overline{\xi}_{z}k\left(y,z\right)=C_{x_{0}}\left\Vert f\right\Vert _{\mathscr{H}}^{2};
\]
and eq. (\ref{eq:ff2}) follows, see also (\ref{eq:fm0}) in \thmref{pm}.\end{proof}
\begin{cor}
Let $V,k,\mathscr{H}$ be as above, i.e., assuming $\delta_{x}\in\mathscr{H}$,
$\forall x\in V$. 
\begin{enumerate}
\item Then $AB:\mathscr{H}\longrightarrow\mathscr{H}$ is a symmetric operator
in $\mathscr{H}$ with dense domain $span\left\{ k_{x}\right\} =\mathscr{D}\left(B\right)$,
and $AB\,k_{x}=\delta_{x}\in\mathscr{H}$. 
\item Moreover, $B^{*}\overline{B}$ is a selfadjoint extension of $AB$
(as an operator in $\mathscr{H}$.)
\end{enumerate}
\end{cor}
\begin{proof}
We have
\[
\underset{\mathscr{H}}{k_{x}}\xrightarrow{\quad B\quad}\underset{l^{2}}{\delta_{x}}\xrightarrow{\quad A\quad}\underset{\mathscr{H}}{\delta_{x}}
\]
and we proved that $\delta_{x}\in dom\left(B^{*}\right)$ in \thmref{BB},
so we conclude that $B^{*}\overline{B}$ extends $AB$. By \corref{ab},
we know that $B^{*}\overline{B}$ is a selfadjoint operator in $\mathscr{H}$. \end{proof}
\begin{thm}
\label{thm:dd}Let $V,k$, and $\mathscr{H}$ be as above, and assume
$\delta_{x}\in\mathscr{H}$, $\forall x\in V$. Then 
\begin{gather}
\delta_{x}\in dom\left(A^{*}\overline{A}\right),\quad\forall x\in V\label{eq:e1}\\
\Updownarrow\nonumber \\
\left\langle \delta_{x},\delta_{\cdot}\right\rangle _{\mathscr{H}}\in l^{2},\quad\forall x\in V.\label{eq:e2}
\end{gather}
\end{thm}
\begin{proof}
Note (\ref{eq:e2}) means 
\begin{equation}
\sum_{y}\left|\left\langle \delta_{x},\delta_{y}\right\rangle _{\mathscr{H}}\right|^{2}<\infty,\quad\forall x\in V.\label{eq:e3}
\end{equation}
Since $\overline{A}\delta_{x_{0}}=\delta_{x_{0}}$, we now show that
(\ref{eq:e1}) holds $\Longleftrightarrow$ (\ref{eq:e3}) is satisfied.
That is, $\exists C_{x_{0}}<\infty$, s.t. 
\begin{equation}
\left|\left\langle A\xi,\delta_{x_{0}}\right\rangle _{\mathscr{H}}\right|^{2}\leq C_{x_{0}}\left\Vert \xi\right\Vert _{l^{2}}^{2}=C_{x_{0}}\sum_{y}\left|\xi_{y}\right|^{2},\quad\forall\left(\xi_{y}\right)\;\mbox{finitely indexed.}\label{eq:e4}
\end{equation}
But $\mbox{LHS}_{\left(\ref{eq:e4}\right)}=\left|\sum\nolimits _{y}\overline{\xi}_{y}\left\langle \delta_{y},\delta_{x_{0}}\right\rangle _{\mathscr{H}}\right|^{2}$,
so (\ref{eq:e4}) holds $\Longleftrightarrow$ $\left[y\longrightarrow\left\langle \delta_{y},\delta_{x_{0}}\right\rangle _{\mathscr{H}}\right]\in l^{2}$
which is the desired conclusion in (\ref{eq:e2}).\end{proof}
\begin{rem}
Note that (\ref{eq:e2}) is not\emph{ }automatic. Examples showing
this? See \lemref{dd} and Examples \ref{exa:ddbm}, \ref{exa:bm}
below.\end{rem}
\begin{lem}
\label{lem:dd}Let $V,k,\mathscr{H}$ be as above, assuming $\delta_{x}\in\mathscr{H}$,
for all $x\in V$. Then 
\[
\left\langle \delta_{x},\delta_{y}\right\rangle _{\mathscr{H}}=\lim_{F}\left(K_{F}^{-1}\right)_{xy},
\]
where the RHS is the inductive limit over the filter of all finite
subsets $F$ of $V$, and 
\[
K_{F}=\left(\left\langle k_{x},k_{y}\right\rangle _{\mathscr{H}}\right)_{\left(x,y\right)\in F\times F}=\left(k\left(x,y\right)\right)_{\left(x,y\right)\in F\times F},
\]
i.e., the Gramian matrix.\end{lem}
\begin{proof}
We have 
\begin{eqnarray*}
\left\langle \delta_{x},\delta_{y}\right\rangle _{\mathscr{H}} & = & \lim_{F}\left\langle \delta_{x},P_{F}\left(\delta_{y}\right)\right\rangle _{\mathscr{H}}\\
 & = & \lim_{F}\left\langle \delta_{x},\sum\nolimits _{s}\left(K_{F}^{-1}\right)_{s\,y}k_{s}\right\rangle _{\mathscr{H}}\;\left(\mbox{Lemma }\ref{lem:proj}\right)\\
 & = & \lim_{F}\left(K_{F}^{-1}\right)_{xy}
\end{eqnarray*}
which is the assertion. 
\end{proof}

\begin{flushleft}
\textbf{Infinite square matrices.}
\par\end{flushleft}
\begin{lem}
Let $k,V$ and $\mathscr{H}\left(=\mathscr{H}\left(k\right)\right)$
be as above, and assume $\delta_{x}\in\mathscr{H}$ for all $x\in V$.
Consider three $\infty\times\infty$ matrices, $D$, $K$, and $C$
as follows:
\begin{equation}
D_{xy}:=\left\langle \delta_{x},\delta_{y}\right\rangle _{\mathscr{H}}\label{eq:m1}
\end{equation}
Set $C$ s.t. 
\begin{equation}
\delta_{x}=\sum_{y}C_{xy}k_{y},\label{eq:m2}
\end{equation}
and let 
\begin{equation}
K_{xy}=k\left(x,y\right).\label{eq:m3}
\end{equation}
Then, 
\begin{equation}
D=CKC^{tr},\label{eq:m4}
\end{equation}
or equivalently, 
\[
D_{xx'}=\left(CKC^{tr}\right)_{xx'},\quad\forall\left(x,x'\right)\in V\times V.
\]

Moreover, let $I=\left(\delta_{xy}\right)_{\left(x,y\right)\in V\times V}$
be the $\infty\times\infty$ identity matrix, then we have
\[
I=CK.
\]
\end{lem}
\begin{proof}
In our discussion of infinite matrices below, for the infinite summations
involved, we are making use of the limit considerations which we made
precise in the proof of \lemref{dd} above.

Apply $\left.k_{z}\right\rangle $ to both sides in (\ref{eq:m2}),
then 
\[
\delta_{xz}=\sum_{y}C_{xy}K_{yz}\Longleftrightarrow I=CK.\;(\mbox{matrix product})
\]
Using $\left\langle \delta_{x},k_{z}\right\rangle _{\mathscr{H}}=\delta_{xz}$,
we get 
\begin{equation}
\left(CKC^{tr}\right)_{xx'}=\sum_{y}\sum_{y'}C_{xy}K_{yy'}C_{x'y'}.\label{eq:m5}
\end{equation}
Also, 
\begin{eqnarray*}
D_{xx'} & \underset{\text{by \ensuremath{\left(\ref{eq:m1}\right)}}}{=} & \left\langle \delta_{x},\delta_{x'}\right\rangle _{\mathscr{H}}\\
 & \underset{\text{by \ensuremath{\left(\ref{eq:m2}\right)}}}{=} & \left\langle \sum\nolimits _{y}C_{xy}k_{y},\sum\nolimits _{y'}C_{x'y'}k_{y'}\right\rangle _{\mathscr{H}}\\
 & = & \sum\nolimits _{y}\sum\nolimits _{y'}C_{xy}C_{x'y'}\left\langle k_{y},k_{y'}\right\rangle _{\mathscr{H}}\\
 & \underset{\text{by \ensuremath{\left(\ref{eq:m3}\right)}}}{=} & \left(CKC^{tr}\right)_{xx'}
\end{eqnarray*}
which is the desired conclusion, see (\ref{eq:m5}).
\end{proof}
In the discussion below we shall consider matrix algebra for ``infinite
square matrices.\textquotedblright{} More precisely, we shall apply
matrix algebra to pairs of matrices where in each matrix factor, both
rows and columns are indexed by the same given countable infinite
set $V$. Nonetheless, matrix multiplication in this context will
require our use of the limit considerations from the proof of \lemref{dd}
above. In other words, the infinite sums entail a limit over filters
of finite subsets of $V$, as discussed in the proof of \lemref{dd}.
\begin{lem}
We have $C=K^{-1}$, or equivalently, 
\[
\left(KC\right)_{xy}=\delta_{xz}\Longleftrightarrow\sum_{z}K_{xz}C_{zy}=\delta_{xy}.
\]
\end{lem}
\begin{proof}
Apply $\left.k_{x'}\right\rangle $ to both sides in (\ref{eq:m2}),
and we get 
\[
\delta_{xx'}=\sum_{y}C_{xy}K_{yx'}=\left(CK\right)_{xx'}.
\]
\end{proof}
\begin{cor}
\label{cor:CD}Assuming now that the given positive definite function
$k$ is real valued, we then get the following: 
\[
D=C^{tr}=\left(K^{tr}\right)^{-1}=K^{-1}.
\]
\end{cor}
\begin{proof}
Note that $K=K^{tr}$ if $k:V\times V\longrightarrow\mathbb{R}$ is
real valued since 
\[
k\left(x,y\right)=\left\langle k_{x},k_{y}\right\rangle =\left\langle k_{y},k_{x}\right\rangle =k\left(y,x\right);
\]
and so $D=\left(K^{tr}\right)^{-1}=K^{-1}$. 
\end{proof}
\textbf{Spectral Theory. }Let $V,k$, and $\mathscr{H}=\mathscr{H}\left(k\right)$
be as above. Let $\left(A,B\right)$ be the associated dual pair of
operators from \corref{ab}. Since $A^{*}\overline{A}$ is selfadjoint
in $l^{2}=l^{2}\left(V\right)$ with dense domain, it has a canonical
$l^{2}$-projection valued measure $P^{\left(A\right)}\left(\cdot\right)$
defined on the Borel $\sigma$-algebra $\mathscr{B}_{+}$ of subsets
of $[0,\infty)$. We set
\begin{equation}
d\mu_{x}^{\left(A\right)}\left(\lambda\right):=\left\Vert P^{\left(A\right)}\left(d\lambda\right)\delta_{x}\right\Vert _{l^{2}}^{2}.\label{eq:a1}
\end{equation}

\begin{lem}
\label{lem:dsp}~
\begin{enumerate}
\item The following conclusions hold for the measure $\mu_{x}^{\left(A\right)}$: 

\begin{enumerate}
\item Moments of order $0,1,$ and $2$:
\begin{align}
 & \;\mu_{x}^{\left(A\right)}\left([0,\infty)\right)=1,\quad\forall x\in V;\label{eq:mm1}\\
 & \int_{0}^{\infty}\lambda\,d\mu_{x}^{\left(A\right)}\left(\lambda\right)=\left\Vert \delta_{x}\right\Vert _{\mathscr{H}}^{2};\;\mbox{and}\label{eq:mm2}\\
 & \int_{0}^{\infty}\lambda^{2}d\mu_{x}^{\left(A\right)}\left(\lambda\right)=\sum_{y\in V}\left|\left\langle \delta_{y},\delta_{x}\right\rangle _{\mathscr{H}}\right|^{2}=\left\Vert A^{*}\delta_{x}\right\Vert _{l^{2}}^{2};\label{eq:mm3}
\end{align}

\item The covariance of the measure $\mu_{x}^{\left(A\right)}$ is:
\begin{equation}
cov\left(\mu_{x}^{\left(A\right)}\right)=\left\Vert A^{*}\delta_{x}\right\Vert _{l^{2}}^{2}-\left\Vert \delta_{x}\right\Vert _{\mathscr{H}}^{4}.\label{eq:mm4}
\end{equation}

\end{enumerate}
\item Moreover, the first moment in (\ref{eq:mm2}) is finite iff $\delta_{x}\in\mathscr{H}$.
The second moment in (\ref{eq:mm3}) is finite iff $\delta_{x}\in dom\left(A^{*}\right)$. 
\end{enumerate}
\end{lem}
\begin{proof}
~

For (\ref{eq:mm1}), we have:
\[
\mu_{x}^{\left(A\right)}\left([0,\infty)\right)=\left\Vert P^{\left(A\right)}\left([0,\infty)\right)\delta_{x}\right\Vert _{l^{2}}^{2}=\left\Vert \delta_{x}\right\Vert _{l^{2}}^{2}=1.
\]

For (\ref{eq:mm2}), we have:
\begin{eqnarray*}
\int_{0}^{\infty}\lambda\,d\mu_{x}^{\left(A\right)}\left(\lambda\right) & \underset{\left(\text{by Cor. \ref{cor:ab}}\right)}{=} & \left\Vert \left(A^{*}\overline{A}\right)^{\frac{1}{2}}\delta_{x}\right\Vert _{l^{2}}^{2}\\
 & \underset{\left(\text{by \ensuremath{\left(\ref{eq:p1}\right)}}\right)}{=} & \left\Vert V\left(A^{*}\overline{A}\right)^{\frac{1}{2}}\delta_{x}\right\Vert _{\mathscr{H}}^{2}\\
 & \underset{\left(\text{by \ensuremath{\left(\ref{eq:p1}\right)}}\right)}{=} & \left\Vert \overline{A}\delta_{x}\right\Vert _{\mathscr{H}}^{2}\\
 & \underset{\left(\text{by Def. \ref{def:dp}}\right)}{=} & \left\Vert \delta_{x}\right\Vert _{\mathscr{H}}^{2}.
\end{eqnarray*}

For (\ref{eq:mm3}), if $\delta_{x}\in dom\left(A^{*}\right)$, then
$A^{*}\delta_{x}\in l^{2}$. Therefore
\[
A^{*}\delta_{x}=\sum_{y\in V}\left\langle \delta_{y},A^{*}\delta_{x}\right\rangle _{l^{2}}\delta_{y},\quad\left(l^{2}\mbox{-convergence}\right)
\]
and 
\[
\left\Vert A^{*}\delta_{x}\right\Vert _{l^{2}}^{2}=\sum_{y\in V}\left|\left\langle \delta_{y},A^{*}\delta_{x}\right\rangle _{l^{2}}\right|^{2}.
\]
But
\begin{equation}
\left\langle \delta_{y},A^{*}\delta_{x}\right\rangle _{l^{2}}=\left\langle A\delta_{y},\delta_{x}\right\rangle _{\mathscr{H}}=\left\langle \delta_{y},\delta_{x}\right\rangle _{\mathscr{H}}.\label{eq:mm5}
\end{equation}
Therefore, with the use of \remref{aa}, we arrive at the following:
\begin{eqnarray*}
\int_{0}^{\infty}\lambda^{2}d\mu_{x}^{\left(A\right)}\left(\lambda\right) & \underset{\left(\text{by Cor. \ref{cor:ab}}\right)}{=} & \left\Vert A^{*}\overline{A}\delta_{x}\right\Vert _{l^{2}}^{2}\\
 & \underset{\left(\text{by \ensuremath{\left(\ref{eq:mm5}\right)}}\right)}{=} & \sum_{y\in V}\left|\left\langle \delta_{y},\delta_{x}\right\rangle _{\mathscr{H}}\right|^{2}.
\end{eqnarray*}

The remaining conclusions in the lemma are now immediate from this.\end{proof}
\begin{cor}
If the equivalent conditions in \thmref{dd} are satisfied, then,
for every $x\in V$, there is a finite positive Borel measure $\mu_{x}$
on $[0,\infty)$ such that 
\begin{equation}
\sum_{y\in V}\left|\left\langle \delta_{y},\delta_{x}\right\rangle _{\mathscr{H}}\right|^{2}=\int_{0}^{\infty}\lambda^{2}\,d\mu_{x}\left(\lambda\right),\quad\forall x\in V.\label{eq:aa1}
\end{equation}
In general, the $\mathscr{H}$-norm of $\delta_{x}$ is finite iff
the first moment of $\mu_{x}$ is finite. \end{cor}
\begin{proof}
Let the condition in the corollary hold. We then make use of the selfadjoint
operator $A^{*}\overline{A}$ from \corref{ab}. We conclude that
$\delta_{x}\in dom\left(A^{*}\overline{A}\right)$, $\forall x\in V$.
Let $P^{\left(A\right)}$ denote the projection valued measure obtained
from the selfadjoint operator $A^{*}\overline{A}$, i.e., 
\begin{equation}
A^{*}\overline{A}=\int_{0}^{\infty}\lambda\,P^{\left(A\right)}\left(\lambda\right)\label{eq:aa2}
\end{equation}
holds on the dense domain $dom\left(A^{*}\overline{A}\right)$; hence
if $\delta_{x}\in dom\left(A^{*}\overline{A}\right)^{\frac{1}{2}}$,
we get 
\[
\left\Vert \left(A^{*}\overline{A}\right)^{\frac{1}{2}}\delta_{x}\right\Vert _{l^{2}}^{2}=\left\langle \overline{A}\delta_{x},\overline{A}\delta_{x}\right\rangle _{\mathscr{H}}=\left\langle \delta_{x},\delta_{x}\right\rangle _{\mathscr{H}}=\left\Vert \delta_{x}\right\Vert _{\mathscr{H}}^{2}.
\]
Now set 
\begin{equation}
\mu_{x}\left(\cdot\right)=\left\Vert P\left(\cdot\right)\delta_{x}\right\Vert _{l^{2}}^{2},
\end{equation}
and substitute into (\ref{eq:aa2}). We get 
\begin{eqnarray*}
\left\Vert \delta_{x}\right\Vert _{\mathscr{H}}^{2} & = & \int_{0}^{\infty}\lambda\left\langle \delta_{x},P\left(d\lambda\right)\delta_{x}\right\rangle _{l^{2}}\\
 & = & \int_{0}^{\infty}\lambda\left\Vert P\left(d\lambda\right)\delta_{x}\right\Vert _{l^{2}}^{2}\\
 & = & \int_{0}^{\infty}\lambda\,d\mu_{x}\left(\lambda\right),
\end{eqnarray*}
which is the remaining conclusion.\end{proof}
\begin{cor}
\label{cor:mu}Let $k,V,\mathscr{H}$, and $P^{\left(A\right)}\left(\cdot\right)$
be as above; i.e., $P^{\left(A\right)}$ is the projection valued
measure of the selfadjoint operator $A^{*}\overline{A}$ in $l^{2}$.
For $x,y\in V$, set 
\begin{equation}
d\mu_{x,y}^{\left(A\right)}\left(\lambda\right)=\left\langle \delta_{x},P^{\left(A\right)}\left(d\lambda\right)\delta_{y}\right\rangle _{l^{2}}.\label{eq:am1}
\end{equation}
Then 
\begin{equation}
\left\langle \delta_{x},\delta_{y}\right\rangle _{\mathscr{H}}=\int_{0}^{\infty}\lambda\,d\mu_{x,y}^{\left(A\right)}\left(\lambda\right).\label{eq:am2}
\end{equation}
\end{cor}
\begin{proof}
By \lemref{dsp} and \remref{aa}, we have
\[
\int_{0}^{\infty}\lambda\,d\mu_{x,y}^{\left(A\right)}\left(\lambda\right)=\left\langle \delta_{x},A^{*}\overline{A}\delta_{y}\right\rangle _{l^{2}}=\left\langle \delta_{x},\delta_{y}\right\rangle _{\mathscr{H}.}
\]

\end{proof}
Let $\left(k,V,\mathscr{H}\right)$ be as in \thmref{BB} and \corref{mu};
and let $\left(A,B\right)$ be the associated symmetric pair. Let
$B^{*}\overline{B}$ be the selfadjoint operator in $\mathscr{H}$,
introduced in \corref{ab}; and let $P^{\left(B\right)}\left(\cdot\right)$
be the corresponding projection valued measure; i.e., $P^{\left(B\right)}\left(S\right)$
is a projection in $\mathscr{H}$, $\forall S\in\mathscr{B}_{+}$
= the Borel $\sigma$-algebra of subsets of $[0,\infty)$. In particular,
\begin{equation}
B^{*}\overline{B}=\int_{0}^{\infty}\lambda\,P^{\left(B\right)}\left(d\lambda\right)\label{eq:bn1}
\end{equation}
holds on the dense domain $dom\left(B^{*}\overline{B}\right)$ in
$\mathscr{H}$. 
\begin{prop}
For $x\in V$, set 
\begin{equation}
d\mu_{x}^{\left(B\right)}\left(\lambda\right)=\left\Vert P^{\left(B\right)}\left(d\lambda\right)k_{x}\right\Vert _{\mathscr{H}}^{2}.\label{eq:bn2}
\end{equation}
Then for the moments of order $0,1$, and $2$, we have:
\begin{align}
 & \mu_{x}^{\left(B\right)}\left([0,\infty)\right)=k\left(x,x\right);\label{eq:bn3}\\
 & \int_{0}^{\infty}\lambda\,d\mu_{x}^{\left(B\right)}\left(\lambda\right)=1,\quad\mbox{and}\label{eq:bn4}\\
 & \int_{0}^{\infty}\lambda^{2}\,d\mu_{x}^{\left(B\right)}\left(\lambda\right)=\left\Vert \delta_{x}\right\Vert _{\mathscr{H}}^{2}.\label{eq:bn5}
\end{align}
\end{prop}
\begin{proof}
We have
\[
\mbox{LHS}_{\left(\ref{eq:bn3}\right)}=\left\Vert k_{x}\right\Vert _{\mathscr{H}}^{2}=k\left(x,x\right).
\]
\begin{eqnarray*}
\mbox{LHS}_{\left(\ref{eq:bn4}\right)} & = & \left\Vert \left(B^{*}\overline{B}\right)^{\frac{1}{2}}k_{x}\right\Vert _{\mathscr{H}}^{2}\\
 & = & \left\Vert W\left(B^{*}\overline{B}\right)^{\frac{1}{2}}k_{x}\right\Vert _{l^{2}}^{2}\quad\left(\mbox{see \ensuremath{\left(\ref{eq:p2}\right)} in Cor. \ref{cor:ab}}\right)\\
 & \underset{\left(\text{by \ensuremath{\left(\ref{eq:p2}\right)}}\right)}{=} & \left\Vert \overline{B}\,k_{x}\right\Vert _{l^{2}}^{2}\\
 & \underset{\left(\text{by \ensuremath{\left(\ref{eq:B1}\right)}}\right)}{=} & \left\Vert \delta_{x}\right\Vert _{l^{2}}^{2}=1.
\end{eqnarray*}

Similarly, 
\begin{eqnarray*}
\mbox{LHS}_{\left(\ref{eq:bn5}\right)} & = & \left\Vert B^{*}\overline{B}\,k_{x}\right\Vert _{\mathscr{H}}^{2}\\
 & \underset{\left(\text{by \ensuremath{\left(\ref{eq:B1}\right)}}\right)}{=} & \left\Vert B^{*}\delta_{x}\right\Vert _{\mathscr{H}}^{2}=\left\Vert A\delta_{x}\right\Vert _{\mathscr{H}}^{2}=\left\Vert \delta_{x}\right\Vert _{\mathscr{H}}^{2}.
\end{eqnarray*}

We have proved the three moments formulas.
\end{proof}

\subsection{Application:\emph{ }Moment analysis of networks with given conductance
function}

Consider a fixed infinite network as specified in Definitions \ref{def:nt}-\ref{def:eh}.
Recall that $c:E\longrightarrow\mathbb{R}_{+}$ is a fixed conductance
function, i.e., $c_{xy}=c_{yx}$, and defined for $\forall\left(xy\right)\in E$.
We write $x\sim y$ iff (Def.) $\left(xy\right)\in E$. Set 
\begin{equation}
c\left(x\right)=\sum_{y\sim x}c_{xy}.\label{eq:an1}
\end{equation}
The sum in (\ref{eq:an1}) may be finite, or infinite. Let $x\in V\backslash\left\{ o\right\} $,
where ``$o$'' is the chosen base-point in the vertex set $V$.
\begin{thm}
\label{thm:mu}Given $\left(V,E,c\right)$ connected, and let $x\in V\backslash\left\{ o\right\} $.
Set $\mu_{x}^{\left(A\right)}\left(\cdot\right)=\left\Vert P^{\left(A\right)}\left(\cdot\right)\delta_{x}\right\Vert _{l^{2}}^{2}$
(see (\ref{eq:a1})). 
\begin{enumerate}
\item \label{enu:u1}We have
\begin{equation}
\delta_{x}\in\mathscr{H}_{E}\Longleftrightarrow c\left(x\right)<\infty,\label{eq:an2}
\end{equation}
and in this case
\begin{equation}
\int_{0}^{\infty}\lambda\,d\mu_{x}^{\left(A\right)}\left(\lambda\right)=\left\Vert \delta_{x}\right\Vert _{\mathscr{H}_{E}}^{2}=c\left(x\right).\label{eq:an3}
\end{equation}

\item \label{enu:u2}Assume (\ref{eq:an2}) holds for all $x\in V\backslash\left\{ o\right\} $,
then 
\begin{equation}
\int_{0}^{\infty}\lambda^{2}d\mu_{x}^{\left(A\right)}\left(\lambda\right)=\left(c\left(x\right)^{2}+\sum\nolimits _{y\sim x}c_{xy}^{2}\right).\label{eq:an4}
\end{equation}

\item \label{enu:u3}For the covariance of $\mu_{x}^{\left(A\right)}$,
we have:
\begin{equation}
cov\left(\mu_{x}^{\left(A\right)}\right)=\sum_{y\sim x}c_{xy}^{2},\label{eq:an5}
\end{equation}
and 
\begin{equation}
cov\left(\mu_{x}^{\left(A\right)}\right)\leq\left\Vert \delta_{x}\right\Vert _{\mathscr{H}_{E}}^{4}.\label{eq:an6}
\end{equation}

\end{enumerate}
\end{thm}
\begin{proof}
With the use of \lemref{dipp} and \propref{glap}, we get the following
formulas for the $\mathscr{H}_{E}$-inner product, see also \defref{eh}:
\begin{equation}
\left\langle \delta_{x},\delta_{y}\right\rangle _{\mathscr{H}_{E}}=\begin{cases}
c\left(x\right) & \mbox{if }y=x\\
-c_{xy} & \mbox{if }y\sim x\\
0 & \mbox{if }y\neq x\;\mbox{and }\left(xy\right)\notin E.
\end{cases}\label{eq:an7}
\end{equation}

The conclusions (\ref{eq:an2}) and (\ref{eq:an3}) are immediate
from this since $\left\Vert \delta_{x}\right\Vert _{\mathscr{H}_{E}}^{2}=c\left(x\right)$
follows from (\ref{eq:an7}).

Conclusion (\ref{eq:an4}) in the theorem follows from (\ref{eq:mm3})
in \lemref{dsp}, and (\ref{eq:an7}) above. Indeed, 
\begin{eqnarray*}
\int_{0}^{\infty}\lambda^{2}d\mu_{x}^{\left(A\right)}\left(\lambda\right) & = & \left\Vert A^{*}\delta_{x}\right\Vert _{l^{2}}^{2}\\
 & = & \sum_{y\in V}\left|\left\langle \delta_{y},\delta_{x}\right\rangle _{\mathscr{H}_{E}}\right|^{2}\\
 & = & c\left(x\right)^{2}+\sum_{y\sim x}c_{xy}^{2}\quad\left(\text{by }\left(\ref{eq:an7}\right)\right)
\end{eqnarray*}
which is the desired conclusion.

For the covariance, we have:
\begin{eqnarray*}
\int_{0}^{\infty}\left|\lambda-c\left(x\right)\right|^{2}d\mu_{x}^{\left(A\right)}\left(\lambda\right) & = & \int_{0}^{\infty}\lambda^{2}d\mu_{x}^{\left(A\right)}\left(\lambda\right)-c\left(x\right)^{2}\\
 & = & \sum_{y\sim x}c_{xy}^{2};\quad\left(\text{by }\left(\ref{eq:an4}\right)\right)
\end{eqnarray*}
thus completing the proof of conclusions (\ref{enu:u1})-(\ref{enu:u2})
in the statement of the Theorem.

Part (\ref{enu:u3}). Since 
\[
\sum_{y\sim x}c_{xy}^{2}\leq\left(\sum_{y\sim x}c_{xy}\right)^{2}=c\left(x\right)^{2}
\]
we get estimate (\ref{eq:an6}) in part (\ref{enu:u3}) from the Theorem.
\end{proof}

\subsection{Discrete sample points for Brownian motion}

We interrupt the general considerations with an example for illustration,
choices of discrete sample points for standard Brownian motion. 
\begin{example}
\label{exa:ddbm}Consider $V:\;0<x_{1}<x_{2}<\cdots<x_{i}<x_{i+1}<\cdots$,
a discrete subset of $\mathbb{R}_{+}$, and set 
\[
k\left(s,t\right)=s\wedge t=\min\left(s,t\right),\quad\forall s,t\in V.
\]
Note that $k$ is the covariance kernel (positive definite) of standard
Brownian motion, restricted to the set $V$. Let $\mathscr{H}\left(=\mathscr{H}\left(k\right)\right)$
be the associated RKHS. (See sect. \ref{sub:bm} for details.)

For each finite subset $F_{n}=\left\{ x_{1},x_{2},\ldots,x_{n}\right\} $
of $V$, we have 
\[
K_{n}=K^{\left(F_{n}\right)}=\begin{bmatrix}x_{1} & x_{1} & x_{1} & \cdots & x_{1}\\
x_{1} & x_{2} & x_{2} & \cdots & x_{2}\\
x_{1} & x_{2} & x_{3} & \cdots & x_{3}\\
\vdots & \vdots & \vdots & \vdots & \vdots\\
x_{1} & x_{2} & x_{3} & \cdots & x_{n}
\end{bmatrix}=\left(x_{i}\wedge x_{j}\right)_{i,j=1}^{n}.
\]
A direct calculation shows that
\[
\left(K_{n}^{-1}\right)=\left[\begin{array}{ccccc}
-\frac{x_{2}}{x_{1}^{2}-x_{1}x_{2}} & \frac{1}{x_{1}-x_{2}} & 0 & 0 & 0\\
\frac{1}{x_{1}-x_{2}} & \frac{x_{3}-x_{1}}{\left(x_{1}-x_{2}\right)\left(x_{2}-x_{3}\right)} & \frac{1}{x_{2}-x_{3}} & 0 & 0\\
0 & \ddots & \ddots & \ddots & 0\\
0 & 0 & \frac{1}{x_{n-1}-x_{n}} & \frac{x_{n+1}-x_{n-1}}{\left(x_{n-1}-x_{n}\right)\left(x_{n}-x_{n+1}\right)} & \frac{1}{x_{n}-x_{n+1}}
\end{array}\right].
\]
It follows that 
\[
\left[x_{j}\longrightarrow\left\langle \delta_{x_{i}},\delta_{x_{j}}\right\rangle \right]\in l^{2},\quad\forall x_{i}\in V.
\]
See \lemref{dd}.
\end{example}

\section{\label{sec:mspa}Spectral theory: A necessary and sufficient condition
for when the symmetric pair is maximal}

We showed in \secref{SPA} that, to every reproducing kernel Hilbert
space $\mathscr{H}$ having a countable discrete set of sample points
of finite $\mathscr{H}$-norm, there is a canonically associated symmetric
pair of operators $\left(A,B\right)$. In the present section we give
a practical necessary and sufficient condition for this symmetric
pair to be maximal.
\begin{thm}
Let $V,k,\mathscr{H}$ be as above, and assume $\delta_{x}\in\mathscr{H}$,
$\forall x\in V$. Let $(A,B)$ be the associated symmetric pair of
operators from \defref{dp} and \corref{ab}. Then TFAE:
\begin{enumerate}
\item \label{enu:ab1}$\overline{A}=B^{*}$ and $\overline{B}=A^{*}$;
\item \label{enu:ab2}The following implication holds:
\[
\left[h\in dom\left(A^{*}\right),\;\left(A^{*}h\right)\left(x\right)=-h\left(x\right),\;\forall x\in V\right]\Longrightarrow h=0.
\]

\end{enumerate}
\end{thm}
\begin{proof}
It is enough to consider one of the two conditions in (\ref{enu:ab1}).
We note that $\overline{B}=A^{*}$ $\Longleftrightarrow$ 
\begin{equation}
\mathscr{G}\left(A^{*}\right)\ominus\mathscr{G}\left(B\right)=0.\label{eq:ab1}
\end{equation}
But 
\begin{gather}
\begin{pmatrix}h\\
A^{*}h
\end{pmatrix}\in\mathscr{G}\left(A^{*}\right)\ominus\mathscr{G}\left(B\right)\nonumber \\
\Updownarrow\nonumber \\
\left\langle \begin{pmatrix}k_{x}\\
\delta_{x}
\end{pmatrix},\begin{pmatrix}h\\
A^{*}h
\end{pmatrix}\right\rangle _{\oplus}=0,\quad\forall x\in V,\label{eq:ab2}
\end{gather}
where $\left\langle \cdot,\cdot\right\rangle _{\oplus}$ is the inner
product in $\begin{matrix}\underset{\oplus}{\mathscr{H}}\\
l^{2}
\end{matrix}$. Now (\ref{eq:ab2}) $\Longleftrightarrow$ 
\begin{gather*}
\left\langle k_{x},h\right\rangle _{\mathscr{H}}+\left\langle \delta_{x},A^{*}h\right\rangle _{l^{2}}=0,\quad\forall x\in V\\
\Updownarrow\\
h\left(x\right)+\left(A^{*}h\right)\left(x\right)\equiv0,\quad\forall x\in V.
\end{gather*}
The desired conclusion (\ref{enu:ab1})$\Longleftrightarrow$ (\ref{enu:ab2})
is now immediate.\end{proof}
\begin{cor}
Let $V,k,\mathscr{H}$ be as above, and assume $\delta_{x}\in\mathscr{H}$,
$\forall x\in V$. Let $\left(A,B\right)$ be the associated dual
pair of operators, with $P^{\left(A\right)}$ and $P^{\left(B\right)}$
the respective projection valued measures. Set $\mu_{x}^{\left(A\right)}$
and $\mu_{x}^{\left(B\right)}$ as in (\ref{eq:a1}) and (\ref{eq:bn2}).
Then we have
\begin{equation}
\int_{0}^{\infty}\lambda\,d\mu_{x}^{\left(A\right)}\left(\lambda\right)=\int_{0}^{\infty}\lambda^{2}\,d\mu_{x}^{\left(B\right)}\left(\lambda\right)\left(=\left\Vert \delta_{x}\right\Vert _{\mathscr{H}}^{2}\right).\label{eq:wa1}
\end{equation}
Moreover, assume $\left(A,B\right)$ is maximal, and $\delta_{x}\in dom\left(B^{*}\overline{B}\right)$;
then 
\begin{equation}
\int_{0}^{\infty}\lambda^{2}d\mu_{x}^{\left(A\right)}\left(\lambda\right)=\int_{0}^{\infty}\lambda\left\Vert P^{\left(B\right)}\left(\lambda\right)\delta_{x}\right\Vert _{\mathscr{H}}^{2}.\label{eq:wa2}
\end{equation}
\end{cor}
\begin{proof}
Eq. (\ref{eq:wa1}) follows from (\ref{eq:mm2}) and (\ref{eq:bn5}).

For (\ref{eq:wa2}), we have 
\begin{eqnarray*}
\int_{0}^{\infty}\lambda^{2}d\mu_{x}^{\left(A\right)}\left(\lambda\right) & \underset{\left(\text{by \ensuremath{\left(\ref{eq:mm3}\right)}}\right)}{=} & \left\Vert A^{*}\delta_{x}\right\Vert _{l^{2}}^{2}=\left\Vert \overline{B}\delta_{x}\right\Vert _{l^{2}}^{2}\\
 & = & \left\langle \delta_{x},B^{*}\overline{B}\delta_{x}\right\rangle _{\mathscr{H}}\\
 & = & \int_{0}^{\infty}\lambda\left\Vert P^{\left(B\right)}\left(\lambda\right)\delta_{x}\right\Vert _{\mathscr{H}}^{2},
\end{eqnarray*}
which is (\ref{eq:wa2}). 
\end{proof}

\section{\label{sec:egs}Sample point-masses in concrete models}

Suppose $V\subset D\subset\mathbb{R}^{d}$ where $V$ is countable
and discrete, but $D$ is open. In this case, we get two kernels:
$k$ on $D\times D$, and $k_{V}:=k\big|_{V\times V}$ on $V\times V$
by restriction. If $x\in V$, then $k_{x}^{\left(V\right)}\left(\cdot\right)=k\left(\cdot,x\right)$
is a function on $V$, while $k_{x}\left(\cdot\right)=k\left(\cdot,x\right)$
is a function on $D$. 

This means that the corresponding RKHSs are different, $\mathscr{H}_{V}$
vs $\mathscr{H}$, where $\mathscr{H}_{V}=$ a RKHS of functions on
$V$, and $\mathscr{H}=$ a RKHS of functions on $D$. 
\begin{lem}
\label{lem:mc1}$\mathscr{H}_{V}$ is isometrically contained in $\mathscr{H}$
via $k_{x}^{\left(V\right)}\longmapsto k_{x}$, $x\in V$. \end{lem}
\begin{proof}
If $F\subset V$ is a finite subset, and $\xi=\xi_{F}$ is a function
on $F$, then 
\[
\left\Vert \sum\nolimits _{x\in F}\xi\left(x\right)k_{x}^{\left(V\right)}\right\Vert _{\mathscr{H}_{V}}=\left\Vert \sum\nolimits _{x\in F}\xi\left(x\right)k_{x}\right\Vert _{\mathscr{H}}.
\]
The desired result follows from this. (See \propref{riso} for the
case of \emph{point-mass samples}.)
\end{proof}
\textbf{Examples. }We are concerned with cases of kernels $k:D\times D\rightarrow\mathbb{R}$
with restriction $k_{V}:V\times V\rightarrow\mathbb{R}$, where $V$
is a countable discrete subset of $D$. Typically, for $x\in V$,
we may have (restriction) $\delta_{x}\big|_{V}\in\mathscr{H}_{V}$,
but $\delta_{x}\notin\mathscr{H}$; indeed this happens for the kernel
$k$ of standard Brownian motion: 

$D=\mathbb{R}_{+}$;

$V=$ an ordered subset $0<x_{1}<x_{2}<\cdots<x_{i}<x_{i+1}<\cdots$,
$V=\left\{ x_{i}\right\} _{i=1}^{\infty}$. 

In this case, we compute $\mathscr{H}_{V}$, and we show that $\delta_{x_{i}}\big|_{V}\in\mathscr{H}_{V}$;
while for $\mathscr{H}_{m}=$ the Cameron-Martin Hilbert space, we
have $\delta_{x_{i}}\notin\mathscr{H}_{m}$. 

Also note that $\delta_{x_{1}}$ has a different meaning with reference
to $\mathscr{H}_{V}$ vs $\mathscr{H}_{m}$. In the first case, it
is simply $\delta_{x_{1}}\left(y\right)=\begin{cases}
1 & y=x_{1}\\
0 & y\in V\backslash\left\{ x_{1}\right\} 
\end{cases}$. In the second case, $\delta_{x_{1}}$ is a Schwartz distribution.
We shall abuse notation, writing $\delta_{x}$ in both cases. 

In the following, we will consider restriction to $V\times V$ of
a special continuous p.d. kernel $k$ on $\mathbb{R}_{+}\times\mathbb{R}_{+}$.
It is $k\left(s,t\right)=s\wedge t=\min\left(s,t\right)$. Before
we restrict, note that the RKHS of this $k$ is the Cameron-Martin
Hilbert space of function $f$ on $\mathbb{R}_{+}$ with distribution
derivative $f'\in L^{2}\left(\mathbb{R}_{+}\right)$, and 
\begin{equation}
\left\Vert f\right\Vert _{\mathscr{H}}^{2}:=\int_{0}^{\infty}\left|f'\left(t\right)\right|^{2}dt<\infty.\label{eq:cm1}
\end{equation}
For details, see below. 
\begin{rem}[Application]
 The Hilbert space given by $\left\Vert \cdot\right\Vert _{\mathscr{H}}^{2}$
in (\ref{eq:cm1}) is called the Cameron-Martin Hilbert space, and,
as noted, it is the RKHS of $k:\mathbb{R}_{+}\times\mathbb{R}_{+}\rightarrow\mathbb{R}:$
$k\left(s,t\right):=s\wedge t$. Now pick a discrete subset $V\subset\mathbb{R}_{+}$;
then Lemma \ref{lem:mc1} states that the RKHS of the $V\times V$
restricted kernel, $k^{\left(V\right)}$ is isometrically embedded
into $\mathscr{H}$, i.e., setting 
\begin{equation}
J^{\left(V\right)}\left(k_{x}^{\left(V\right)}\right)=k_{x},\quad\forall x\in V;\label{eq:cm2}
\end{equation}
$J^{\left(V\right)}$ extends by ``closed span'' to an isometry
$\mathscr{H}_{V}\xrightarrow{J^{\left(V\right)}}\mathscr{H}$. It
further follows from the lemma, that the range of $J^{\left(V\right)}$
may have infinite co-dimension. 

Note that $P_{V}:=J^{\left(V\right)}\left(J^{\left(V\right)}\right)^{*}$
is the projection onto the range of $J^{\left(V\right)}$. The ortho-complement
is as follow: 
\begin{equation}
\mathscr{H}\ominus\mathscr{H}_{V}=\left\{ \psi\in\mathscr{H}\:\big|\:\psi\left(x\right)=0,\;\forall x\in V\right\} .\label{eq:cm3}
\end{equation}
\end{rem}
\begin{example}
Let $k$ and $k^{\left(V\right)}$ be as in (\ref{eq:cm2}), and set
$V:=\pi\mathbb{Z}_{+}$, i.e., integer multiples of $\pi$. Then easy
generators of wavelet functions \cite{BJ02} yield non-zero functions
$\psi$ on $\mathbb{R}_{+}$ such that 
\begin{equation}
\psi\in\mathscr{H}\ominus\mathscr{H}_{V}.\label{eq:cm4}
\end{equation}
More precisely, 
\begin{equation}
0<\int_{0}^{\infty}\left|\psi'\left(t\right)\right|^{2}dt<\infty,\label{eq:cm5}
\end{equation}
where $\psi'$ is the distribution (weak) derivative; and 
\begin{equation}
\psi\left(n\pi\right)=0,\quad\forall n\in\mathbb{Z}_{+}.\label{eq:cm6}
\end{equation}
An explicit solution to (\ref{eq:cm4})-(\ref{eq:cm6}) is 
\begin{equation}
\psi\left(t\right)=\prod_{n=1}^{\infty}\cos\left(\frac{t}{2^{n}}\right)=\frac{\sin t}{t},\quad\forall t\in\mathbb{R}.\label{eq:cm7}
\end{equation}
From this, one easily generates an infinite-dimensional set of solutions. 
\end{example}

\subsection{\label{sub:bm}Sample points in Brownian motion}

Consider the covariance function of standard Brownian motion $B_{t}$,
$t\in[0,\infty)$, i.e., a Gaussian process $\left\{ B_{t}\right\} $
with mean zero and covariance function 
\begin{equation}
\mathbb{E}\left(B_{s}B_{t}\right)=s\wedge t=\min\left(s,t\right).\label{eq:bm1}
\end{equation}
We now show that the restriction of (\ref{eq:bm1}) to $V\times V$
for an ordered subset (we fix such a set $V$):
\begin{equation}
V:\;0<x_{1}<x_{2}<\cdots<x_{i}<x_{i+1}<\cdots\label{eq:bm2}
\end{equation}
has the discrete mass property (Definition \ref{def:dmp}). 

Set $\mathscr{H}_{V}=RKHS(k\big|_{V\times V})$, 
\begin{equation}
k_{V}\left(x_{i},x_{j}\right)=x_{i}\wedge x_{j}.\label{eq:bm3}
\end{equation}
We consider the set $F_{n}=\left\{ x_{1},x_{2},\ldots,x_{n}\right\} $
of finite subsets of $V$, and 
\begin{equation}
K_{n}=k^{\left(F_{n}\right)}=\begin{bmatrix}x_{1} & x_{1} & x_{1} & \cdots & x_{1}\\
x_{1} & x_{2} & x_{2} & \cdots & x_{2}\\
x_{1} & x_{2} & x_{3} & \cdots & x_{3}\\
\vdots & \vdots & \vdots & \vdots & \vdots\\
x_{1} & x_{2} & x_{3} & \cdots & x_{n}
\end{bmatrix}=\left(x_{i}\wedge x_{j}\right)_{i,j=1}^{n}.\label{eq:bm4}
\end{equation}
We will show that condition \ref{enu:d3} in Theorem \ref{thm:del}
holds for $k_{V}$. For this, we must compute all the determinants,
$D_{n}=\det\left(K_{F}\right)$ etc. ($n=\#F$), see Corollary \ref{cor:sp}.
\begin{lem}
~ 
\begin{equation}
D_{n}=\det\left(\left(x_{i}\wedge x_{j}\right)_{i,j=1}^{n}\right)=x_{1}\left(x_{2}-x_{1}\right)\left(x_{3}-x_{2}\right)\cdots\left(x_{n}-x_{n-1}\right).\label{eq:bm5}
\end{equation}
\end{lem}
\begin{proof}
Induction. In fact, 
\[
\begin{bmatrix}x_{1} & x_{1} & x_{1} & \cdots & x_{1}\\
x_{1} & x_{2} & x_{2} & \cdots & x_{2}\\
x_{1} & x_{2} & x_{3} & \cdots & x_{3}\\
\vdots & \vdots & \vdots & \vdots & \vdots\\
x_{1} & x_{2} & x_{3} & \cdots & x_{n}
\end{bmatrix}\sim\begin{bmatrix}x_{1} & 0 & 0 & \cdots & 0\\
0 & x_{2}-x_{1} & 0 & \cdots & 0\\
0 & 0 & x_{3}-x_{2} & \cdots & 0\\
\vdots & \vdots & \vdots & \ddots & \vdots\\
0 & \cdots & 0 & \cdots & x_{n}-x_{n-1}
\end{bmatrix},
\]
unitary equivalence in finite dimensions.
\end{proof}

\begin{lem}
Let 
\begin{equation}
\zeta_{\left(n\right)}:=K_{n}^{-1}\left(\delta_{x_{1}}\right)\left(\cdot\right)\label{eq:bm7}
\end{equation}
be as in  (\ref{eq:pd8}), so that 
\begin{equation}
\left\Vert P_{F_{n}}\left(\delta_{x_{1}}\right)\right\Vert _{\mathscr{H}_{V}}^{2}=\zeta_{\left(n\right)}\left(x_{1}\right).\label{eq:bm8}
\end{equation}
Then, 
\begin{eqnarray*}
\zeta_{\left(1\right)}\left(x_{1}\right) & = & \frac{1}{x_{1}}\\
\zeta_{\left(n\right)}\left(x_{1}\right) & = & \frac{x_{2}}{x_{1}\left(x_{2}-x_{1}\right)},\quad\text{for}\;n=2,3,\ldots,
\end{eqnarray*}
and 
\[
\left\Vert \delta_{x_{1}}\right\Vert _{\mathscr{H}_{V}}^{2}=\frac{x_{2}}{x_{1}\left(x_{2}-x_{1}\right)}.
\]
\end{lem}
\begin{proof}
A direct computation shows the $\left(1,1\right)$ minor of the matrix
$K_{n}^{-1}$ is
\begin{equation}
D'_{n-1}=\det\left(\left(x_{i}\wedge x_{j}\right)_{i,j=2}^{n}\right)=x_{2}\left(x_{3}-x_{2}\right)\left(x_{4}-x_{3}\right)\cdots\left(x_{n}-x_{n-1}\right)\label{eq:bm6}
\end{equation}
and so 
\begin{eqnarray*}
\zeta_{\left(1\right)}\left(x_{1}\right) & = & \frac{1}{x_{1}},\quad\mbox{and}\\
\zeta_{\left(2\right)}\left(x_{1}\right) & = & \frac{x_{2}}{x_{1}\left(x_{2}-x_{1}\right)}\\
\zeta_{\left(3\right)}\left(x_{1}\right) & = & \frac{x_{2}\left(x_{3}-x_{2}\right)}{x_{1}\left(x_{2}-x_{1}\right)\left(x_{3}-x_{2}\right)}=\frac{x_{2}}{x_{1}\left(x_{2}-x_{1}\right)}\\
\zeta_{\left(4\right)}\left(x_{1}\right) & = & \frac{x_{2}\left(x_{3}-x_{2}\right)\left(x_{4}-x_{3}\right)}{x_{1}\left(x_{2}-x_{1}\right)\left(x_{3}-x_{2}\right)\left(x_{4}-x_{3}\right)}=\frac{x_{2}}{x_{1}\left(x_{2}-x_{1}\right)}\\
 & \vdots
\end{eqnarray*}
The result follows from this, and from Corollary \ref{cor:proj1}.\end{proof}
\begin{cor}
\label{cor:proj}$P_{F_{n}}\left(\delta_{x_{1}}\right)=P_{F_{2}}\left(\delta_{x_{1}}\right)$,
$\forall n\geq2$. Therefore, 
\begin{equation}
\delta_{x_{1}}\in\mathscr{H}_{V}^{\left(F_{2}\right)}:=span\{k_{x_{1}}^{\left(V\right)},k_{x_{2}}^{\left(V\right)}\}
\end{equation}
and
\begin{equation}
\delta_{x_{1}}=\zeta_{\left(2\right)}\left(x_{1}\right)k_{x_{1}}^{\left(V\right)}+\zeta_{\left(2\right)}\left(x_{2}\right)k_{x_{2}}^{\left(V\right)}
\end{equation}
where 
\[
\zeta_{\left(2\right)}\left(x_{i}\right)=K_{2}^{-1}\left(\delta_{x_{1}}\right)\left(x_{i}\right),\;i=1,2.
\]
Specifically, 
\begin{eqnarray}
\zeta_{\left(2\right)}\left(x_{1}\right) & = & \frac{x_{2}}{x_{1}\left(x_{2}-x_{1}\right)}\\
\zeta_{\left(2\right)}\left(x_{2}\right) & = & \frac{-1}{x_{2}-x_{1}};
\end{eqnarray}
and 
\begin{equation}
\left\Vert \delta_{x_{1}}\right\Vert _{\mathscr{H}_{V}}^{2}=\frac{x_{2}}{x_{1}\left(x_{2}-x_{1}\right)}.\label{eq:dn}
\end{equation}
\end{cor}
\begin{proof}
Follows from the lemma. Note that 
\[
\zeta_{n}\left(x_{1}\right)=\left\Vert P_{F_{n}}\left(\delta_{x_{1}}\right)\right\Vert _{\mathscr{H}}^{2}
\]
and $\zeta_{\left(1\right)}\left(x_{1}\right)\leq\zeta_{\left(2\right)}\left(x_{1}\right)\leq\cdots$,
since $F_{n}=\left\{ x_{1},x_{2},\ldots,x_{n}\right\} $. In particular,
$\frac{1}{x_{1}}\leq\frac{x_{2}}{x_{1}\left(x_{2}-x_{1}\right)}$,
which yields (\ref{eq:dn}). \end{proof}
\begin{rem}
We showed that $\delta_{x_{1}}\in\mathscr{H}_{V}$, $V=\left\{ x_{1}<x_{2}<\cdots\right\} \subset\mathbb{R}_{+}$,
with the restriction of $s\wedge t$ = the covariance kernel of Brownian
motion. 

The same argument also shows that $\delta_{x_{i}}\in\mathscr{H}_{V}$
when $i>1$. We only need to modify the index notation from the case
of the proof for $\delta_{x_{1}}\in\mathscr{H}_{V}$. The details
are sketched below.

Fix $V=\left\{ x_{i}\right\} _{i=1}^{\infty}$, $x_{1}<x_{2}<\cdots$,
then 
\[
P_{F_{n}}\left(\delta_{x_{i}}\right)=\begin{cases}
0 & \text{if \ensuremath{n<i-1}}\\
\sum_{s=1}^{n}\left(K_{F_{n}}^{-1}\delta_{x_{i}}\right)\left(x_{s}\right)k_{x_{s}} & \text{if \ensuremath{n\geq i}}
\end{cases}
\]
and 
\[
\left\Vert P_{F_{n}}\left(\delta_{x_{i}}\right)\right\Vert _{\mathscr{H}}^{2}=\begin{cases}
0 & \text{if \ensuremath{n<i-1}}\\
\frac{1}{x_{i}-x_{i-1}} & \text{if \ensuremath{n=i}}\\
\frac{x_{i+1}-x_{i-1}}{\left(x_{i}-x_{i-1}\right)\left(x_{i+1}-x_{i}\right)} & \text{if \ensuremath{n>i}}
\end{cases}
\]
\textbf{Conclusion.} 
\begin{eqnarray}
\delta_{x_{i}} & \in & span\left\{ k_{x_{i-1}}^{\left(V\right)},k_{x_{i}}^{\left(V\right)},k_{x_{i+1}}^{\left(V\right)}\right\} ,\quad\mbox{and}\\
\left\Vert \delta_{x_{i}}\right\Vert _{\mathscr{H}}^{2} & = & \frac{x_{i+1}-x_{i-1}}{\left(x_{i}-x_{i-1}\right)\left(x_{i+1}-x_{i}\right)}.
\end{eqnarray}
\end{rem}
\begin{cor}
Let $V\subset\mathbb{R}_{+}$ be countable. If $x_{a}\in V$ is an
accumulation point (from $V$), then $\left\Vert \delta_{a}\right\Vert _{\mathscr{H}_{V}}=\infty$. \end{cor}
\begin{example}
\label{exa:bm}An illustration for $0<x_{1}<x_{2}<x_{3}<x_{4}$: 
\begin{eqnarray*}
P_{F}\left(\delta_{x_{3}}\right) & = & \sum_{y\in F}\zeta^{\left(F\right)}\left(y\right)k_{y}\left(\cdot\right)\\
\zeta^{\left(F\right)} & = & K_{F}^{-1}\delta_{x_{3}}\;.
\end{eqnarray*}
That is, 
\[
\underset{\left(K_{F}\left(x_{i},x_{j}\right)\right)_{i,j=1}^{4}}{\underbrace{\begin{bmatrix}x_{1} & x_{1} & x_{1} & x_{1}\\
x_{1} & x_{2} & x_{2} & x_{2}\\
x_{1} & x_{2} & x_{3} & x_{3}\\
x_{1} & x_{2} & x_{3} & x_{4}
\end{bmatrix}}}\begin{bmatrix}\zeta^{\left(F\right)}\left(x_{1}\right)\\
\zeta^{\left(F\right)}\left(x_{2}\right)\\
\zeta^{\left(F\right)}\left(x_{3}\right)\\
\zeta^{\left(F\right)}\left(x_{4}\right)
\end{bmatrix}=\begin{bmatrix}0\\
0\\
1\\
0
\end{bmatrix}
\]
and 
\begin{eqnarray*}
\zeta^{\left(F\right)}\left(x_{3}\right) & = & \frac{x_{1}\left(x_{2}-x_{1}\right)\left(x_{4}-x_{2}\right)}{x_{1}\left(x_{2}-x_{1}\right)\left(x_{3}-x_{2}\right)\left(x_{4}-x_{3}\right)}\\
 & = & \frac{x_{4}-x_{2}}{\left(x_{3}-x_{2}\right)\left(x_{4}-x_{3}\right)}=\left\Vert \delta_{x_{3}}\right\Vert _{\mathscr{H}}^{2}.
\end{eqnarray*}

\end{example}

\begin{example}[Sparse sample-points]
Let $V=\left\{ x_{i}\right\} _{i=1}^{\infty}$, where 
\[
x_{i}=\frac{i\left(i-1\right)}{2},\quad i\in\mathbb{N}.
\]
It follows that $x_{i+1}-x_{i}=i$, and so 
\[
\left\Vert \delta_{x_{i}}\right\Vert _{\mathscr{H}}^{2}=\frac{x_{i+1}-x_{i-1}}{\left(x_{i}-x_{i-1}\right)\left(x_{i+1}-x_{i}\right)}=\frac{2i-1}{\left(i-1\right)i}\xrightarrow[i\rightarrow\infty]{}0.
\]
We conclude that $\left\Vert \delta_{x_{i}}\right\Vert _{\mathscr{H}}\xrightarrow[i\rightarrow\infty]{}0$
if the set $V=\left\{ x_{i}\right\} _{i=1}^{\infty}\subset\mathbb{R}_{+}$
is sparse. 
\end{example}
Now, some general facts:
\begin{lem}
Let $k:V\times V\rightarrow\mathbb{C}$ be p.d., and let $\mathscr{H}$
be the corresponding RKHS. If $x_{1}\in V$, and if $\delta_{x_{1}}$
has a representation as follows:
\begin{equation}
\delta_{x_{1}}=\sum_{y\in V}\zeta^{\left(x_{1}\right)}\left(y\right)k_{y},\label{eq:pr1}
\end{equation}
then
\begin{equation}
\left\Vert \delta_{x_{1}}\right\Vert _{\mathscr{H}}^{2}=\zeta^{\left(x_{1}\right)}\left(x_{1}\right).\label{eq:pr2}
\end{equation}
\end{lem}
\begin{proof}
Substitute both sides of (\ref{eq:pr1}) into $\left\langle \delta_{x_{1}},\cdot\right\rangle _{\mathscr{H}}$
where $\left\langle \cdot,\cdot\right\rangle _{\mathscr{H}}$ denotes
the inner product in $\mathscr{H}$. \end{proof}
\begin{example}[Application]
 Suppose $V=\cup_{n}F_{n}$, $F_{n}\subset F_{n+1}$, where each
$F_{n}\in\mathscr{F}\left(V\right)$, then if $x_{1}\in F_{n}$, we
have 
\begin{equation}
P_{F_{n}}\left(\delta_{x_{1}}\right)=\sum_{y\in F_{n}}\left\langle x_{1},K_{F_{n}}^{-1}y\right\rangle _{l^{2}}k_{y}\label{eq:pr3}
\end{equation}
and 
\begin{equation}
\left\Vert P_{F_{n}}\left(\delta_{x_{1}}\right)\right\Vert _{\mathscr{H}}^{2}=\left\langle x_{1},K_{F_{n}}^{-1}x_{1}\right\rangle _{l^{2}}=\left(K_{F_{n}}^{-1}\delta_{x_{1}}\right)\left(x_{1}\right)\label{eq:pr4}
\end{equation}
and the expression $\left\Vert P_{F_{n}}\left(\delta_{x_{1}}\right)\right\Vert _{\mathscr{H}}^{2}$
is monotone in $n$, i.e., 
\[
\left\Vert P_{F_{n}}\left(\delta_{x_{1}}\right)\right\Vert _{\mathscr{H}}^{2}\leq\left\Vert P_{F_{n+1}}\left(\delta_{x_{1}}\right)\right\Vert _{\mathscr{H}}^{2}\leq\cdots\leq\left\Vert \delta_{x_{1}}\right\Vert _{\mathscr{H}}^{2}
\]
with 
\[
\sup_{n\in\mathbb{N}}\left\Vert P_{F_{n}}\left(\delta_{x_{1}}\right)\right\Vert _{\mathscr{H}}^{2}=\lim_{n\rightarrow\infty}\left\Vert P_{F_{n}}\left(\delta_{x_{1}}\right)\right\Vert _{\mathscr{H}}^{2}=\left\Vert \delta_{x_{1}}\right\Vert _{\mathscr{H}}^{2}.
\]

\end{example}
For other applications of reproducing kernel Hilbert spaces to the
analysis of Gaussian processes, see e.g., \cite{MR2573134,JP13}.
\begin{question}
Let $k:\mathbb{R}^{d}\times\mathbb{R}^{d}\rightarrow\mathbb{R}$ be
positive definite, and let $V\subset\mathbb{R}^{d}$ be a countable
discrete subset, e.g., $V=\mathbb{Z}^{d}$. When does $k\big|_{V\times V}$
have the \uline{discrete mass} property? 
\end{question}
Examples of the affirmative, or not, will be discussed below.

\subsection{Discrete RKHSs from restrictions}

Let $D:=[0,\infty)$, and $k:D\times D\rightarrow\mathbb{R}$, with
\[
k\left(x,y\right)=x\wedge y=\min\left(x,y\right).
\]
Restrict to $V:=\left\{ 0\right\} \cup\mathbb{Z}_{+}\subset D$, i.e.,
consider 
\[
k^{\left(V\right)}=k\big|_{V\times V}.
\]
$\mathscr{H}\left(k\right)$: Cameron-Martin Hilbert space, consisting
of functions $f\in L^{2}\left(\mathbb{R}\right)$ s.t. 
\[
\int_{0}^{\infty}\left|f'\left(x\right)\right|^{2}dx<\infty,\quad f\left(0\right)=0.
\]
$\mathscr{H}_{V}:=\mathscr{H}\left(k_{V}\right)$. Note that 
\[
f\in\mathscr{H}\left(k_{V}\right)\Longleftrightarrow\sum_{n}\left|f\left(n\right)-f\left(n+1\right)\right|^{2}<\infty.
\]

\begin{lem}
We have $\delta_{n}=2k_{n}-k_{n+1}-k_{n-1}\in\mathscr{H}_{V}$. \end{lem}
\begin{proof}
Introduce the discrete Laplacian $\Delta=\Delta_{c}$ (see (\ref{eq:L1})),
i.e., 
\[
\left(\Delta f\right)\left(x\right)=\sum_{y\sim x}c_{xy}\left(f\left(x\right)-f\left(y\right)\right),
\]
defined for all functions $f$ on $V=\left\{ 0\right\} \cup\mathbb{Z}_{+}\subset D$,
and $c:E\rightarrow\mathbb{R}_{+}$ is the corresponding conductance.
Setting $c\equiv1$, we get 
\[
\left(\Delta f\right)\left(n\right)=2f\left(n\right)-f\left(n-1\right)-f\left(n+1\right).
\]
But, by (\ref{eq:L3}) in \propref{glap}, we have $\Delta k_{n}=\delta_{n}$,
and the assertion of the lemma follows from this. Note that 
\[
\left\langle 2k_{n}-k_{n+1}-k_{n-1},k_{m}\right\rangle _{\mathscr{H}_{V}}=\left\langle \delta_{n},k_{m}\right\rangle _{\mathscr{H}_{V}}=\delta_{n,m}.
\]
\end{proof}
\begin{rem}
The same argument as in the proof of the lemma shows (\emph{mutatis
mutandis}) that any ordered discrete countable infinite subset $V\subset[0,\infty)$
yields 
\[
\mathscr{H}_{V}:=\mathscr{H}\left(k\big|_{V\times V}\right)
\]
as a RKHS which is discrete in that (\defref{dmp}) if $V=\left\{ x_{i}\right\} _{i=1}^{\infty}$,
$x_{i}\in\mathbb{R}_{+}$, then $\delta_{x_{i}}\in\mathscr{H}_{V}$,
$\forall i\in\mathbb{N}$. \end{rem}
\begin{proof}
Fix vertices $V=\left\{ x_{i}\right\} _{i=1}^{\infty}$, 
\begin{equation}
0<x_{1}<x_{2}<\cdots<x_{i}<x_{i+1}<\infty,\quad x_{i}\rightarrow\infty.
\end{equation}
Assign conductance 
\begin{equation}
c_{i,i+1}=c_{i+1,i}=\frac{1}{x_{i+1}-x_{i}}\left(=\frac{1}{\text{dist}}\right)
\end{equation}
Let 
\begin{eqnarray}
\left(\Delta f\right)\left(x_{i}\right) & = & \left(\frac{1}{x_{i+1}-x_{i}}+\frac{1}{x_{i}-x_{i-1}}\right)f\left(x_{i}\right)\nonumber \\
 &  & -\frac{1}{x_{i}-x_{i-1}}f\left(x_{i-1}\right)-\frac{1}{x_{i+1}-x_{i}}f\left(x_{i+1}\right)
\end{eqnarray}
Equivalently, 
\begin{equation}
\left(\Delta f\right)\left(x_{i}\right)=\left(c_{i,i+1}+c_{i,i-1}\right)f\left(x_{i}\right)-c_{i,i-1}f\left(x_{i-1}\right)-c_{i,i+1}f\left(x_{i+1}\right).\label{eq:glap}
\end{equation}

Then, with (\ref{eq:glap}) we have: 
\[
\Delta k_{x_{i}}=\delta_{x_{i}}
\]
where $k\left(\cdot,\cdot\right)=$ restriction of $s\wedge t$ from
$[0,\infty)\times[0,\infty)$ to $V\times V$; and therefore 
\begin{equation}
\delta_{x_{i}}=\left(c_{i,i+1}+c_{i,i-1}\right)k_{x_{i}}-c_{i,i+1}k_{x_{i+1}}-c_{i,i-1}k_{x_{i-1}}\in\mathscr{H}_{V}
\end{equation}
as the right-side in the last equation is a finite sum. Note that
now the RKHS is 
\[
\mathscr{H}_{V}=\left\{ f:V\rightarrow\mathbb{C}\:\big|\:\sum_{i=1}^{\infty}c_{i,i+1}\left|f\left(x_{i+1}\right)-f\left(x_{i}\right)\right|^{2}<\infty\right\} .
\]

\end{proof}

\subsection{\label{sub:Bridge}Brownian bridge}

Let $D:=\left(0,1\right)=$ the open interval $0<t<1$, and set 
\begin{equation}
k_{bridge}\left(s,t\right):=s\wedge t-st;\label{eq:bb1}
\end{equation}
then (\ref{eq:bb1}) is the covariance function for the Brownian bridge
$B_{bri}\left(t\right)$, i.e., 
\begin{equation}
B_{bri}\left(0\right)=B_{bri}\left(1\right)=0\label{eq:bb2}
\end{equation}

\begin{equation}
B_{bri}\left(t\right)=\left(1-t\right)B\left(\frac{t}{1-t}\right),\quad0<t<1;\label{eq:bb3}
\end{equation}
where $B\left(t\right)$ is Brownian motion; see \lemref{mc1}.

\begin{figure}[H]
\includegraphics[width=0.5\columnwidth]{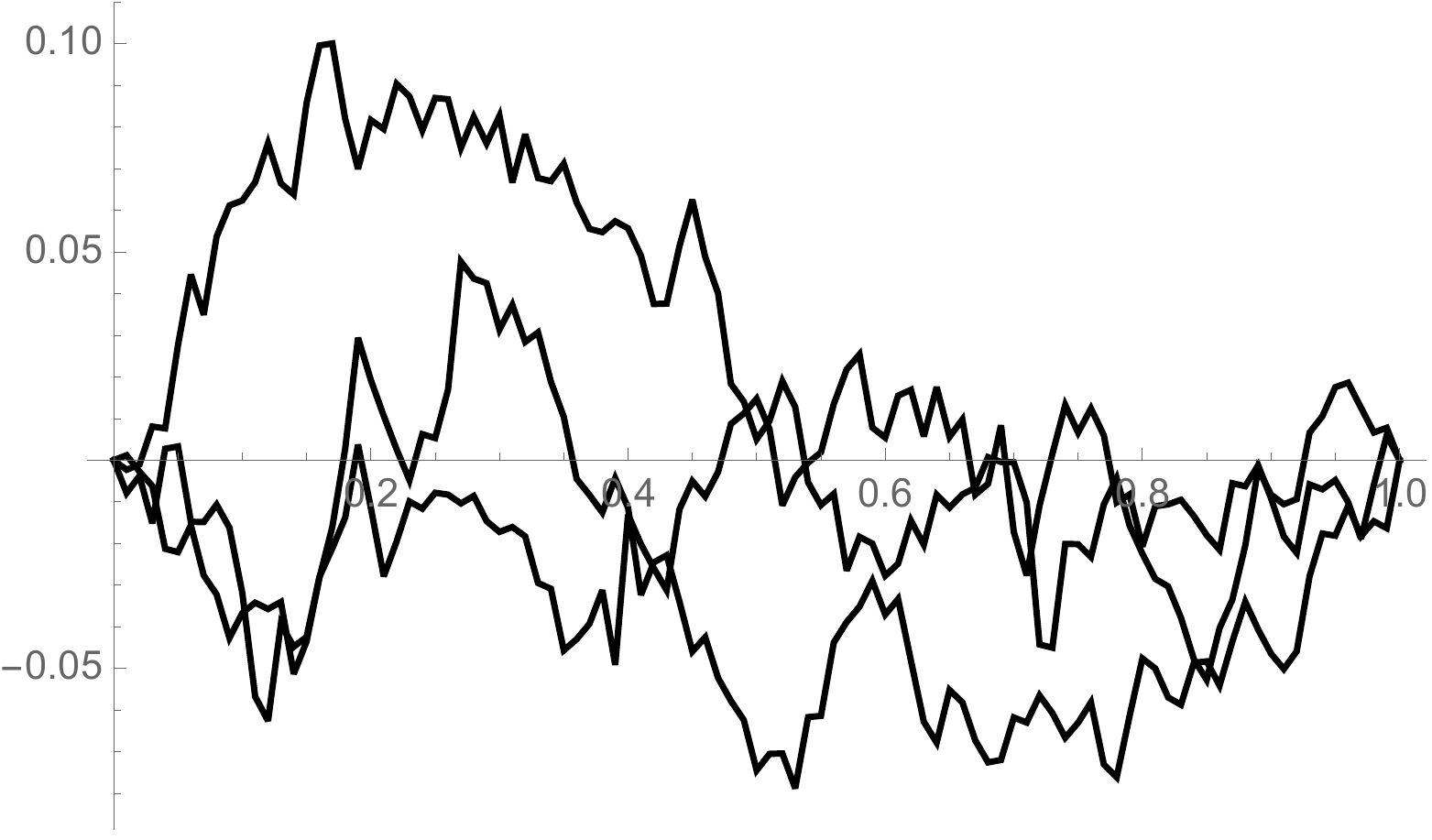}

\caption{\label{fig:bb}Brownian bridge $B_{bri}\left(t\right)$, a simulation
of three sample paths of the Brownian bridge.}
\end{figure}

The corresponding Cameron-Martin space is now 
\begin{equation}
\mathscr{H}_{bri}=\left\{ f\;\mbox{on}\:\left[0,1\right];f'\in L^{2}\left(0,1\right),f\left(0\right)=f\left(1\right)=0\right\} \label{eq:bb4}
\end{equation}
with 
\begin{equation}
\left\Vert f\right\Vert _{\mathscr{H}_{bri}}^{2}:=\int_{0}^{1}\left|f'\left(s\right)\right|^{2}ds<\infty.\label{eq:bb5}
\end{equation}

If $V=\left\{ x_{i}\right\} _{i=1}^{\infty}$, $x_{1}<x_{2}<\cdots<1$,
is the discrete subset of $D$, then we have for $F_{n}\in\mathscr{F}\left(V\right)$,
$F_{n}=\left\{ x_{1},x_{2},\cdots,x_{n}\right\} $, 
\begin{equation}
K_{F_{n}}=\left(k_{bridge}\left(x_{i},x_{j}\right)\right)_{i,j=1}^{n},\label{eq:bb6}
\end{equation}
see (\ref{eq:bb1}), and 
\begin{equation}
\det K_{F_{n}}=x_{1}\left(x_{2}-x_{1}\right)\cdots\left(x_{n}-x_{n-1}\right)\left(1-x_{n}\right).\label{eq:bb7}
\end{equation}

As a result, we get $\delta_{x_{i}}\in\mathscr{H}_{V}^{\left(bri\right)}$
for all $i$, and 
\[
\left\Vert \delta_{x_{i}}\right\Vert _{\mathscr{H}_{V}^{\left(bri\right)}}^{2}=\frac{x_{i+1}-x_{i-1}}{\left(x_{i+1}-x_{i}\right)\left(x_{i}-x_{i-1}\right)}.
\]
Note $\lim_{x_{i}\rightarrow1}\left\Vert \delta_{x_{i}}\right\Vert _{\mathscr{H}_{V}^{\left(bri\right)}}^{2}=\infty$.

\subsection{\label{sub:bion}Binomial RKHS}

It is possible to associate a positive definite kernel (see \defref{bino})
to the standard binomial coefficients. In this section we outline
the properties of this kernel and its reproducing kernel Hilbert space.
Among the conclusions is that in this RKHS, the point-masses have
infinite $\mathscr{H}$-norm.
\begin{defn}
\label{def:bino}Let $V=\mathbb{Z}_{+}\cup\left\{ 0\right\} $; and
\[
k_{b}\left(x,y\right):=\sum_{n=0}^{x\wedge y}\binom{x}{n}\binom{y}{n},\quad\left(x,y\right)\in V\times V.
\]
where $\binom{x}{n}=\frac{x\left(x-1\right)\cdots\left(x-n+1\right)}{n!}$
denotes the standard binomial coefficient from the binomial expansion.

Let $\mathscr{H}=\mathscr{H}\left(k_{b}\right)$ be the corresponding
RKHS. Set 
\begin{equation}
e_{n}\left(x\right)=\begin{cases}
\binom{x}{n} & \text{if \ensuremath{n\leq x}}\\
0 & \text{if \ensuremath{n>x}}.
\end{cases}\label{eq:b1}
\end{equation}
\end{defn}
\begin{lem}[\cite{AJ15}]
\label{lem:b1}~
\begin{enumerate}[label=(\roman{enumi})]
\item $e_{n}\left(\cdot\right)\in\mathscr{H}$, $n\in V$; 
\item $\left\{ e_{n}\right\} _{n\in V}$ is an orthonormal basis (ONB) in
the Hilbert space $\mathscr{H}$. 
\item Set $F_{n}=\left\{ 0,1,2,\ldots,n\right\} $, and 
\begin{equation}
P_{F_{n}}=\sum_{k=0}^{n}\left|e_{k}\left\rangle \right\langle e_{k}\right|\label{eq:b2}
\end{equation}
or equivalently 
\begin{equation}
P_{F_{n}}f=\sum_{k=0}^{n}\left\langle e_{k},f\right\rangle _{\mathscr{H}}e_{k}\,.\label{eq:b3}
\end{equation}

\end{enumerate}

then, 
\begin{enumerate}[resume]
\item[(iv)] Formula (\ref{eq:b3}) is well defined for all functions $f:V\rightarrow\mathbb{C}$,
$f\in\mathscr{F}unc\left(V\right)$. \end{enumerate}
\begin{enumerate}
\item[(v)] Given $f\in\mathscr{F}unc\left(V\right)$; then 
\begin{equation}
f\in\mathscr{H}\Longleftrightarrow\sum_{k=0}^{\infty}\left|\left\langle e_{k},f\right\rangle _{\mathscr{H}}\right|^{2}<\infty;\label{eq:b4}
\end{equation}
and, in this case, 
\[
\left\Vert f\right\Vert _{\mathscr{H}}^{2}=\sum_{k=0}^{\infty}\left|\left\langle e_{k},f\right\rangle _{\mathscr{H}}\right|^{2}.
\]

\end{enumerate}
\end{lem}

Fix $x_{1}\in V$, then we shall apply Lemma \ref{lem:b1} to the
function $f_{1}=\delta_{x_{1}}$ (in $\mathscr{F}unc\left(V\right)$),
$f_{1}\left(y\right)=\begin{cases}
1 & \text{if \ensuremath{y=x_{1}}}\\
0 & \text{if \ensuremath{y\neq x_{1}}.}
\end{cases}$ 
\begin{thm}
\label{thm:bino}We have 
\[
\left\Vert P_{F_{n}}\left(\delta_{x_{1}}\right)\right\Vert _{\mathscr{H}}^{2}=\sum_{k=x_{1}}^{n}\binom{k}{x_{1}}^{2}.
\]

\end{thm}
The proof of the theorem will be subdivided in steps; see below. 
\begin{lem}[\cite{AJ15}]
~
\begin{enumerate}[label=(\roman{enumi})]
\item \label{enu:b1}For $\forall m,n\in V$, such that $m\leq n$, we
have 
\begin{equation}
\delta_{m,n}=\sum_{j=m}^{n}\left(-1\right)^{m+j}\binom{n}{j}\binom{j}{m}.\label{eq:b5}
\end{equation}

\item \label{enu:b2}For all $n\in\mathbb{Z}_{+}$, the inverse of the following
lower triangle matrix is this: With (see Figure \ref{fig:L}) 
\begin{equation}
L_{xy}^{\left(n\right)}=\begin{cases}
\binom{x}{y} & \text{if \ensuremath{y\leq x\leq n}}\\
0 & \text{if \ensuremath{x<y}}
\end{cases}\label{eq:b6}
\end{equation}
 we have:
\begin{equation}
\left(L^{\left(n\right)}\right)_{xy}^{-1}=\begin{cases}
\left(-1\right)^{x-y}\binom{x}{y} & \text{if \ensuremath{y\leq x\leq n}}\\
0 & \text{if \ensuremath{x<y}}.
\end{cases}\label{eq:b7}
\end{equation}
 
\end{enumerate}

Notation: The numbers in (\ref{eq:b7}) are the entries of the matrix
$\left(L^{\left(n\right)}\right)^{-1}$. 

\end{lem}
\begin{proof}
In rough outline, \ref{enu:b2} follows from \ref{enu:b1}.\end{proof}
\begin{cor}
\label{cor:bino}Let $k_{b}$, $\mathscr{H}$, and $n\in\mathbb{Z}_{+}$
be as above with the lower triangle matrix $L_{n}$. Set 
\begin{equation}
K_{n}\left(x,y\right)=k_{b}\left(x,y\right),\quad\left(x,y\right)\in F_{n}\times F_{n},\label{eq:b8}
\end{equation}
i.e., an $\left(n+1\right)\times\left(n+1\right)$ matrix. 
\begin{enumerate}[label=(\roman{enumi})]
\item Then $K_{n}$ is invertible with 
\begin{equation}
K_{n}^{-1}=\left(L_{n}^{tr}\right)^{-1}\left(L_{n}\right)^{-1};\label{eq:b9}
\end{equation}
an $(\text{upper triangle})\times(\text{lower triangle})$ factorization. 
\item For the diagonal entries in the $\left(n+1\right)\times\left(n+1\right)$
matrix $K_{n}^{-1}$, we have:
\[
\left\langle x,K_{n}^{-1}x\right\rangle _{l^{2}}=\sum_{k=x}^{n}\binom{k}{x}^{2}
\]

\end{enumerate}

Conclusion\textbf{:} Since 
\begin{equation}
\left\Vert P_{F_{n}}\left(\delta_{x_{1}}\right)\right\Vert _{\mathscr{H}}^{2}=\left\langle x_{1},K_{n}^{-1}x_{1}\right\rangle _{\mathscr{H}}\label{eq:b11}
\end{equation}
for all $x_{1}\in F_{n}$, we get 
\begin{eqnarray}
\left\Vert P_{F_{n}}\left(\delta_{x_{1}}\right)\right\Vert _{\mathscr{H}}^{2} & = & \sum_{k=x_{1}}^{n}\binom{k}{x_{1}}^{2}\nonumber \\
 & = & 1+\binom{x_{1}+1}{x_{1}}^{2}+\binom{x_{1}+2}{x_{1}}^{2}+\cdots+\binom{n}{x_{1}}^{2};\label{eq:b12}
\end{eqnarray}
and therefore, 
\[
\left\Vert \delta_{x_{1}}\right\Vert _{\mathscr{H}}^{2}=\sum_{k=x_{1}}^{\infty}\binom{k}{x_{1}}^{2}=\infty.
\]
In other words, no $\delta_{x}$ is in $\mathscr{H}$.

\end{cor}
\begin{figure}[H]
\[
L^{\left(n\right)}=\begin{bmatrix}1 & 0 & 0 & 0 & \cdots & \cdots & 0 & \cdots & 0 & 0\\
1 & 1 & 0 & 0 & \cdots & \cdots & 0 & \cdots & 0 & 0\\
1 & 2 & 1 & 0 &  &  & \vdots &  & \vdots & \vdots\\
1 & 3 & 3 & 1 & \ddots &  & \vdots &  & \vdots & \vdots\\
\vdots & \vdots & \vdots & \vdots & \ddots &  & \vdots &  & \vdots & \vdots\\
\vdots & \vdots & \vdots & \vdots &  & 1 & 0 &  & \vdots & \vdots\\
1 & \cdots & \binom{x}{y} & \binom{x}{y+1} & \cdots & * & 1 & \ddots & \vdots & \vdots\\
\vdots & \vdots & \vdots & \vdots &  &  &  & \ddots & 0 & \vdots\\
\vdots & \vdots & \vdots & \vdots &  &  &  &  & 1 & 0\\
1 & \cdots & \binom{n}{y} & \binom{n}{y+1} & \cdots & \cdots & \cdots & \cdots & n & 1
\end{bmatrix}
\]

\caption{\label{fig:L}The matrix $L_{n}$ is simply a truncated Pascal triangle,
arranged to fit into a lower triangular matrix.}
\end{figure}

\subsection{Classical RKHSs with point-mass samples and interpolation}
\begin{defn}
\label{def:pm}Let $\mathscr{H}$ be a RKHS (or a relative RKHS) defined
from a positive definite kernel $k\left(x,y\right)$, $\left(x,y\right)\in V\times V$.
A (discrete) subset $S\subset V$ is said to be a set of \emph{point-mass
samples} iff the span of $\left\{ k_{x}\mid x\in S\right\} $ is dense
in $\mathscr{H}$. \end{defn}
\begin{lem}
Let $V,k$ and $\mathscr{H}$ be as stated in \defref{pm}, and assume
$S\subset V$ is a countable discrete subset; then $S$ is a point-mass
sample set if $\exists\epsilon\in\mathbb{R}_{+}$ such that 
\begin{equation}
\sum_{s\in S}\left|f\left(s\right)\right|^{2}\geq\epsilon\left\Vert f\right\Vert _{\mathscr{H}}^{2},\quad\forall f\in\mathscr{H}.\label{eq:pm1}
\end{equation}
\end{lem}
\begin{proof}
To show that $span\left\{ k_{s}\mid s\in S\right\} $ is dense in
$\mathscr{H}$, we need only verify that if 
\[
0=f\left(s\right)=\left\langle k_{s},f\right\rangle _{\mathscr{H}},\quad\forall s\in S,
\]
then $f\equiv0$ in $\mathscr{H}$. But this conclusion is immediate
from the estimate (\ref{eq:pm1}).\end{proof}
\begin{prop}
\label{prop:riso}Let $k:V\times V\longrightarrow\mathbb{C}$ be positive
definite, and let $\mathscr{H}$ be the corresponding RKHS. Let $S\subset V$
be a set of point-mass samples. Then the following holds for the restricted
kernel function $k^{\left(S\right)}$, defined by 
\begin{equation}
k^{\left(S\right)}\left(s,t\right):=k\left(s,t\right),\quad\forall\left(s,t\right)\in S\times S:\label{eq:r1}
\end{equation}

If $\mathscr{H}^{\left(S\right)}$ denotes the RKHS of $k^{\left(S\right)}$,
then the assignment
\begin{equation}
W^{\left(S\right)}k_{s}^{\left(S\right)}:=k_{s},\quad s\in S\label{eq:r2}
\end{equation}
extends by linearity and norm-closure to an \uline{isometry} $W^{\left(S\right)}$
of $\mathscr{H}^{\left(S\right)}$ \uline{onto} $\mathscr{H}$. \end{prop}
\begin{proof}
For finite subsets $F\subset S$, and $\left\{ \xi_{s}\right\} _{s\in F}$,
we have the following:
\begin{equation}
\left\Vert \sum\nolimits _{s\in F}\xi_{s}k_{s}^{\left(S\right)}\right\Vert _{\mathscr{H}^{\left(S\right)}}=\left\Vert \sum\nolimits _{s\in F}\xi_{s}k_{s}\right\Vert _{\mathscr{H}}.
\end{equation}
Hence $W^{\left(S\right)}$ in (\ref{eq:r2}) extends by linearity
and closure to an isometry $W^{\left(S\right)}:\mathscr{H}^{\left(S\right)}\longrightarrow\mathscr{H}$.
(Also see \lemref{mc1}.)

Since the range $ran\left(W^{\left(S\right)}\right)=\left\{ W^{\left(S\right)}h^{\left(S\right)}\mid h^{\left(S\right)}\in\mathscr{H}^{\left(S\right)}\right\} $
is automatically closed in $\mathscr{H}$, we need only prove that
$ran\left(W^{\left(S\right)}\right)$ is dense in $\mathscr{H}$;
i.e., $\mathscr{H}\ominus ran\left(W^{\left(S\right)}\right)=0$.
By (\ref{eq:r1})-(\ref{eq:r2}), we must prove that, if $f\in\mathscr{H}$,
and 
\[
f\left(s\right)=\left\langle k_{s},f\right\rangle _{\mathscr{H}}=0,\quad\forall s\in S,
\]
then $f=0$ in $\mathscr{H}$. But the last conclusion is immediate
from the condition on the set $S$ from \defref{pm}.\end{proof}
\begin{prop}[Interpolation]
 Let $\left(V,k,\mathscr{H}\right)$ be as above, and let $S\subset V$
be a \uline{sample set}, i.e., satisfying the condition in \defref{pm}.
Let $\left(A,B\right)$ be the associated dual pair of operators;
see \lemref{split}. 

Then the following \uline{interpolation formula} holds for $f\in\mathscr{H}$:
\begin{equation}
f=\sum_{s\in S}\left(A^{*}f\right)\left(s\right)k_{s},\label{eq:ip1}
\end{equation}
or equivalently, 
\begin{equation}
f\left(x\right)=\sum_{s\in S}\left(A^{*}f\right)\left(s\right)k\left(x,s\right),\quad\forall x\in V;\label{eq:ip2}
\end{equation}
and convergence in (\ref{eq:ip1})-(\ref{eq:ip2}) holds iff
\begin{equation}
\underset{\left(s,t\right)\in S\times S}{\sum\sum}\overline{\left(A^{*}f\right)\left(s\right)}\left(A^{*}f\right)\left(t\right)k\left(s,t\right)<\infty.\label{eq:ip3}
\end{equation}
When (\ref{eq:ip3}) holds, then $\left\Vert f\right\Vert _{\mathscr{H}}^{2}=\mbox{LHS}_{\left(\ref{eq:ip3}\right)}$. \end{prop}
\begin{proof}
Suppose $f=\sum_{s\in S}C_{s}k_{s}\left(\in\mathscr{H}\right)$ is
a finite sum-representation; then the coefficients $\left\{ C_{s}\right\} _{s\in S}$
are unique. Indeed, if $t\in S$, then 
\[
\left\langle \delta_{t},f\right\rangle _{\mathscr{H}}=\sum_{s\in S}C_{s}\underset{\delta_{t,s}}{\underbrace{\left\langle \delta_{t},k_{s}\right\rangle _{\mathscr{H}}}=C_{t};\;\mbox{and}}
\]
\[
\left\langle \delta_{t},f\right\rangle _{\mathscr{H}}=\left\langle A\delta_{t},f\right\rangle _{\mathscr{H}}=\left\langle \delta_{t},A^{*}f\right\rangle _{l^{2}}=\left(A^{*}f\right)\left(t\right),\;t\in S.
\]
 
\end{proof}
The next example shows that there are many RKHSs $\left(k,\mathscr{H},V\right)$
which satisfy the condition in \defref{pm} for a variety of countably
discrete sample sets $S\subset V$; but nonetheless, the point-masses
$\delta_{x}$ are not in $\mathscr{H}$, i.e., $\left\Vert \delta_{x}\right\Vert _{\mathscr{H}}=\infty$
for all $x\in V$.
\begin{example}
Let $V=\mathbb{R}$ , and let $\mathscr{H}=\big\{ f\in L^{2}\left(\mathbb{R}\right)\mid\mbox{suppt}\widehat{f}\subset\left[-\frac{1}{2},\frac{1}{2}\right]\big\}$,
where $\widehat{f}$ denotes the Fourier transform. This RKHS is said
to be a band-limited Hilbert space. It is known that then 
\begin{equation}
k\left(x,y\right)=\frac{\sin\pi\left(x-y\right)}{\pi\left(x-y\right)},\quad x,y\in\mathbb{R}\label{eq:sh1}
\end{equation}
is a positive definite kernel turning $\mathscr{H}$ into a RKHS. 

Moreover, $\left\{ k_{n}\mid n\in\mathbb{Z}\right\} $ is then a set
of point-mass samples. In fact, $\left\{ k_{n}\right\} _{n\in\mathbb{Z}}$
is an orthonormal basis (ONB) in $\mathscr{H}$, and 
\begin{equation}
f\left(x\right)=\sum_{n\in\mathbb{Z}}k\left(n,x\right)f\left(n\right),\quad\forall f\in\mathscr{H}\label{eq:sh2}
\end{equation}
holds. Note that (\ref{eq:sh2}) is Shannon's sampling formula, and
we have 
\[
\left\Vert f\right\Vert _{\mathscr{H}}^{2}=\sum_{n\in\mathbb{Z}}\left|f\left(n\right)\right|^{2},\quad\forall f\in\mathscr{H}.
\]
It is also known that in addition to $\mathbb{Z}$, there are many
other choices of discrete point-mass samples for $\mathscr{H}$. 

For related, recent studies of sampling spaces corresponding to irregular
distribution of sample-points, see e.g., \cite{MR3285408,MR3074509}.\end{example}
\begin{acknowledgement*}
The co-authors thank the following colleagues for helpful and enlightening
discussions: Professors Sergii Bezuglyi, Ilwoo Cho, Paul Muhly, Myung-Sin
Song, Wayne Polyzou, and members in the Math Physics seminar at The
University of Iowa.

\bibliographystyle{amsalpha}
\bibliography{ref}
\end{acknowledgement*}

\end{document}